\documentclass[12pt]{article} \usepackage{fullpage}
\usepackage{lineno}

\usepackage{nicefrac}
\usepackage{url}
\usepackage{color}
\usepackage[american]{babel}
\usepackage{graphicx}
\usepackage{version}
\usepackage{subcaption}
\usepackage[textsize=tiny]{todonotes}
\usepackage{enumitem}
\setitemize{itemsep=-2pt}
\setenumerate{itemsep=-2pt}
\usepackage{multirow}

\graphicspath{{figures/}}

\usepackage{amsthm,amsmath,amssymb}
\newtheorem{theorem}{Theorem}
\newtheorem{lemma}{Lemma}[section]
\newtheorem{corollary}[lemma]{Corollary}
\newtheorem{definition}[lemma]{Definition}

\newtheorem{claim}[lemma]{Claim}

\newtheorem{observation}[lemma]{Observation}

\newcommand{\student}[1]{}
\newcommand{\postdoc}[1]{}
\newcommand{\calP}{{\ensuremath{\cal P}}}

\newcommand{\calT}{{\ensuremath{\cal T}}}
\newcommand{\calB}{{\ensuremath{\cal B}}}
\newcommand{\calC}{{\ensuremath{\cal C}}}
\DeclareMathOperator{\odd}{odd}
\DeclareMathOperator{\comp}{comp}
\newcommand{\GammaP}{\ensuremath{\Gamma^\times}}

\date{}
\title{Large Matchings in Maximal 1-planar graphs}
\author{Therese Biedl\thanks{Corresponding author.   David R.~Cheriton School of Computer
Science, University of Waterloo, Waterloo, Ontario N2L 3G1, Canada. {\it biedl@uwaterloo.ca}.
Supported by NSERC.   
} \and John Wittnebel
}

\begin{document}

\maketitle
\begin{abstract}
It is well-known that every maximal planar graph has a matching
of size at least $\tfrac{n+8}{3}$ if $n\geq 14$.  
In this paper, we investigate
similar matching-bounds for maximal {\em 1-planar} graphs, i.e.,
graphs that can be drawn such that every edge has at most one
crossing.
In particular we show that every 3-connected simple-maximal 1-planar graph
has a matching of size at least $\tfrac{2n+6}{5}$; the bound decreases
to 
$\tfrac{3n+14}{10}$ if the graph need not be 3-connected.   
We also give (weaker) bounds when the graph comes with a fixed 1-planar drawing
or is not simple.  All our bounds are tight in the
sense that some graph that satisfies the restrictions has no bigger matching.
\end{abstract}


\section{Introduction}

Matchings are one of the oldest and best-studied problems in graph
theory, see for example the extensive reviews of matching theory in
\cite{Berge73,LP86}.
We focus here on matchings in graphs with 
special drawings.  In particular, a graph is called planar
if it can be drawn without crossing in the plane (detailed definitions
are below).  Nishizeki and Baybars \cite{NB79}
showed that every planar graph with $n\geq X$ vertices has a matching
of size at least $Yn+Z$, where $X,Y,Z$ depend on the minimum degree $\delta$
and the connectivity $\kappa$ of the graph (they explore all possibilities
of $\delta$ and $\kappa$). Their bounds are tight.

To name one specific bound, Nishizeki and Baybars proved that
every 3-connected planar graph with minimum degree 3 has a matching
of size at least $\tfrac{n+8}{3}$ if $n\geq 14$.  This in particular
implies that every simple-maximal planar graph with $n\geq 14$ has a matching of 
this size, since such graphs are 3-connected and hence have minimum
degree 3.  The latter result was re-proved (with a different technique)
in \cite{BDD+04}.

The goal of this paper is to develop similar results for maximal 1-planar graphs,
for which we first need to clarify what we mean by `maximal 1-planar'.   Typically
`maximal' means that we cannot add edges and stay in the class, but do we begin with a fixed
drawing or can we consider all drawings?   Are we allowed to have loops and parallel edges
under some restrictions?  This choice of definition
of `maximal' affects the size of a maximum matching, so we will give bounds both 
for \emph{simple-saturated 1-planar drawings} (where no edge can be added without violating 1-planarity
or simplicity) and for \emph{simple-maximal 1-planar graphs} (where all
possible 1-planar drawings must be simple-saturated).   
In contrast to planar graphs, simple-saturated 1-planar drawings
need not represent a 3-connected graph, so we will distinguish further by whether the graph
is 3-connected or not since this again affects the size of a maximum matching.
Table~\ref{ta:results} shows the results that we achieve; for all these graph classes we prove
a lower bound on the matching size, and the lower bound is tight for some graph.

\def\mywidth{30mm}
\begin{table}[ht]\centering
\begin{tabular}{|c||cc|cc|}
\hline
\multicolumn{5}{|c|}{Size of a maximum matching} \\
\hline
\hline
& \multicolumn{2}{|c|}{3-connected} & 
\multicolumn{2}{|c|}{not 3-connected}  \\
\hline
\multirow{3}{32mm}{simple-maximal 1-planar graph} 
& & \multirow{4}{\mywidth}{\includegraphics[width=\mywidth,page=16]{tight.pdf}}
& & \multirow{4}{\mywidth}{\includegraphics[width=\mywidth,page=17]{tight.pdf}} \\[1.5ex]
& $\tfrac{2n+6}{5}\approx 0.4n$ &
& $\tfrac{3n+14}{10}\approx 0.3n$ &
\\
& Theorem~\ref{thm:3connGraph} &
& Theorem~\ref{thm:not3connGraph} & \\
& & & & \\[2.5ex]
& Figure~\ref{fig:3connGraph_tight} &
& Figure~\ref{fig:not3connGraph_tight} & \\
\hline
\multirow{4}{32mm}{simple-saturated 1-planar drawing}
& & \multirow{4}{\mywidth}{\includegraphics[width=\mywidth,page=18]{tight.pdf}}
& & \multirow{4}{\mywidth}{\includegraphics[width=\mywidth,page=19]{tight.pdf}} \\[1.5ex]
& $\tfrac{n+4}{3}\approx 0.33n$ & 
& $\tfrac{n+6}{4}\approx 0.25n$ &
\\
& Theorem~\ref{thm:3connDrawing} &
& Theorem~\ref{thm:not3connDrawing}  & \\
& & & & \\[2.5ex]
& Figure~\ref{fig:3connDrawing_tight}  &
& Figure~\ref{fig:not3connDrawing_tight} & \\
\hline
\end{tabular}
\caption{Results in this paper.  We list the theorem 
that proves the lower bound and an abstract description of the graph
where it is tight; the listed figure shows all details of the graph.
} 
\label{ta:results}
\end{table}

We also briefly study graphs that are not simple (this makes a difference to the matching-size since
then `maximal' can be achieved with fewer non-parallel edges), and here again prove tight bounds on
the matching-size if the graphs have no loops.

As for prior work on matching in 1-planar graphs, we know of two results.   First, in a generalization of
the results by Nishizeki and Baybars \cite{NB79}, we studied 1-planar graphs
with minimum degree $k$ (for $k=3,\dots,7$) and provided lower bounds on a
matching size; for $k=3,4,5$ these are tight \cite{BW20-1pMin}.   Second,
Fabrici et al.~showed that any 4-connected simple-maximal 1-planar graphs has
a Hamiltonian cycle \cite{FHM+20} (in fact `simple-maximal'
can be relaxed to `at all crossings the endpoints induce $K_4$').  
Therefore it has a matching of size $\lfloor \tfrac{n}{2} \rfloor$, which is of
course tight.    

\paragraph{Other related work:} Our proof of the matching bounds relies on
the Tutte-Berge-formula \cite{Berge57,Berge73}, which relates the size
of a maximum matching to $\max_{S\subset V} \{\odd(G{\setminus} S)-|S|)\}$,
where $\odd(G{\setminus} S)$ denotes the number of connected components of
$G{\setminus} S$ that have odd cardinality.   However, in nearly all our proofs
we actually bound $\max_S \{\comp(G{\setminus} S)-|S|\}$, 
where $\comp(G{\setminus} S)\geq \odd(G{\setminus S})$ is the number of connected components of $G{\setminus} S$,
i.e., counting components of both even and odd cardinality.

There are many results that relate $\comp(G{\setminus} S)$ to $|S|$. 
The \emph{toughness of a graph} 
\cite{Chv73} is
$\max_{\emptyset \neq S\subset V} \{ \comp(G{\setminus} S) / |S| \}$.   Many results have 
been found for the toughness, see an overview by Bauer et al.~\cite{BBS06}.
Another related concept is \emph{$\ell$-connectivity} \cite{CKLL84,Oel87}, 
which is defined to be the minimum $|S|$ for which $\comp(G{\setminus} S)\geq \ell$.
See a survey by Li and Wei \cite{LW15} for more on this and related concepts
of generalized connectivity.

While both toughness and $\ell$-connectivity deal with the relationship
between $\comp(G{\setminus} S)$ and $|S|$, the toughness maximizes the ratio
between the two, and the $\ell$-connectivity minimizes $|S|$ for a given
value of $\comp(G{\setminus} S)$. Neither of them immediately implies
any results for the difference between the two, and so to our knowledge
the results in this paper (which bound $\comp(G{\setminus} S)-|S|$ for
some classes of 1-planar graphs) are new.

\section{Preliminaries}
\label{sec:definitions}

Let $G=(V,E)$ be a graph with $n$ vertices. We assume familiarity
with basic terms in graph theory such as connectedness; see e.g.~\cite{Die12} for details.
A graph is called {\em simple} it it has neither a loop nor a parallel edge.
Throughout most of the paper our input graph is connected, simple and has $n\geq 3$.
However, we sometimes add parallel edges to the input-graph (we never add loops), and 
we sometimes consider subgraphs
that may be disconnected or have very few vertices, and so our claims specifically permit 
disconnected non-simple graphs with $n\leq 2$ unless stated otherwise.

A \emph{matching} $M$ of $G$ is a set of edges for which no two
edges have a common endpoint.   
A vertex is \emph{matched} if it is incident to an edge in $M$ and
\emph{unmatched} otherwise.   We write $\mu(G)$ for the size of 
a maximum matching of $G$.

A \emph{connected component} of a graph $G$ is a maximal subgraph that
is connected; it is called {\em odd} if it has an odd number of vertices.    We write $\comp(G)$ and
$\odd(G)$ for the number of connected components and odd connected components of $G$,
respectively.   We will use these terms mostly for a subgraph $G{\setminus}S$
obtained by deleting the vertices of a set $S$ and their incident edges.
The Tutte-Berge
formula \cite{Berge57,Berge73} famously connects the number of odd components to the size of a
maximum matching.

\begin{theorem}[Tutte-Berge formula]
\label{thm:TutteBerge}
We have $\mu(G)=\tfrac{1}{2} \max_{S\subseteq V} \{|V|+|S|-\odd(G{\setminus}S)\}$.
Equivalently, any maximum matching has
$\max_{S\subseteq V} \{\odd(G{\setminus} S) - |S|\}$ unmatched vertices.
\end{theorem}

A {\em cutting set} is a vertex-set $S$ with
$\comp(G{\setminus} S)>\comp(G)$; it is 
called a {\em cut-vertex} if $|S|=1$ and a {\em cutting pair} if $|S|=2$.
Graph $G$ is called {\em $k$-connected} if it has no cutting set of size $k-1$ or less.

\paragraph{Planar and 1-planar graphs:}
A {\em drawing} $\Gamma$ of a graph $G$ assigns vertices to points in $\mathbb{R}^2$ and edges
to curves in $\mathbb{R}^2$ such that edge-curves connect the corresponding endpoints.
All drawings are assumed to be {\em good} (see e.g.~\cite{Schaefer2013}), which means that 
(a) no two vertex-points coincide and no edge-curve intersects a vertex-point
	except at its two ends;
(b) if two edge-curves intersect at a point $p$ that is not a common endpoint, then they properly cross at $p$;
(c) if three or more edge-curves intersect in a point $p$, then $p$ is a common endpoint of the curves;
(d) if the curves of two edges $e,e'$ intersect twice at points $p\neq p'$, then $e,e'$ are parallel edges
	and $p,p'$ are their endpoints; and
(e) if the curve of an edge $e$ self-intersects at point $p$, then $e$ is a loop and $p$ is its endpoint.

We usually identify the graph-theoretic object (vertex, edge) with the geometric object that
represents it (point, curve).  
A drawing is called
{\em $k$-planar} if every edge participates in at most $k$ crossings; in this paper all
drawings are 1-planar and sometimes we restrict the attention to 0-planar
({\em planar}) drawings.    
A graph is called \emph{planar/$1$-planar} if it
has a planar/1-planar drawing.
We sometimes abuse the term ``drawing'' also for its underlying graph;
in particular, we can speak of a 1-planar drawing as being 3-connected or simple.

For the following definitions fix a planar drawing $\Gamma$ of a (possibly disconnected) graph.
The {\em faces} of $\Gamma$  are the connected regions of $\mathbb{R}^2{\setminus} \Gamma$.
For a face $F$, let $\comp(F)$ be the number of connected components of its boundary.
The \emph{face-boundary} is the boundary of $F$, viewed as a collection of circuits in the graph 
(some of those circuits may consist of just one vertex, or of a single edge visited twice).
Define the \emph{degree} $\deg(F)$ of $F$ to be $m_F+2\comp(F)-2$, where $m_F$ is number of edge-incidences
in the face-boundary (repeatedly visited edges count twice).
For example, the face in Figure~\ref{fig:deg5} has $m_F=3$ and $\comp(F)=2$ (one circuit consists
of a single vertex), hence it has degree 5.%
\footnote{This definition of degree may seem unusual, but is needed to make Lemma~\ref{lem:fd}
correct even for disconnected graphs.   For simple graphs, an equivalent definition is to set
$\deg(F)$ to be the number $d$ such that $F$ can be split
into $d{-}2$ triangles by inserting edges.   But our graphs are not always simple.}
The following formula is well-known (for example
it is used in \cite{BDD+04}), but we give a proof of it in the appendix since we use it for disconnected graphs
that may be small or non-simple, and we are not aware of a proof that shows that the formula extends to this case.

\begin{lemma}
\label{lem:fd}
Let $\Gamma$ be a planar drawing with $n$ vertices, and let $f_d$ be the number
of faces of degree $d$.   Then $\sum_d (d{-}2)f_d=2n-4$, even if $\Gamma$
is non-simple or disconnected or $n\leq 2$.
\end{lemma}

We sometimes use the term \emph{deg-$d$ face} for a face that has degree $d$,
especially if $d\in \{2,3\}$.   
A {\em triangulated} planar drawing is one where all faces have degree 3.
We specifically {\em permit} parallel and loops edges in a triangulated graph,
as long as they are not drawn as deg-2 or deg-1 face.
A {\em bigon} is a face $F$ that is bounded by a simple 2-cycle, i.e., $F$ is incident
to two distinct parallel edges. (If $n\geq 3$ then all deg-2 faces are bigons).  

A planar drawing is called {\em simple-maximal planar} if it is simple and
we cannot add an edge to it without either destroying planarity or simplicity.
It is folklore that any simple-maximal planar drawing $\Gamma$ is triangulated and
3-connected. Whitney's theorem
\cite{Whitney32} states that therefore the underlying graph $G$ has a unique planar drawing 
in the sense that all planar drawings of $G$ have the same face-boundaries.

For the following definitions fix a 1-planar drawing $\Gamma$.
We call an edge {\em crossed} if it contains a crossing and {\em uncrossed} otherwise.
The {\em planarization} $\GammaP$ of $\Gamma$ is the planar
drawing obtained by replacing every crossing with a dummy-vertex of degree 4.  The
{\em cells} of $\Gamma$ are the faces of its planarization $\GammaP$.
The {\em corners} of a cell of $\Gamma$ are the vertices of the corresponding
face of $\GammaP$; corners are vertices or crossings of $\Gamma$.   
A cell is called {\em uncrossed} if all its corners
are vertices, and {\em crossed} otherwise.   
Most terms for planar drawings naturally carry over to 1-planar drawings $\Gamma$ via
the planarization $\GammaP$. For example, the \emph{degree} and \emph{cell-boundary} of a cell in $\Gamma$
is the degree/face-boundary of the corresponding face in $\GammaP$,  
and $\Gamma$ is called {\em triangulated} if $\GammaP$
is triangulated.  

Assume that $\Gamma$ has no loops and consider one circuit of a cell-boundary.   
This must contain at least one vertex $z$ (otherwise
an edge crosses itself).   Furthermore, walking from $z$ along the boundary until we revisit $z$, we must
encounter at least one other vertex $z'$, otherwise we either have a loop, or 
exactly one crossing (then two edges incident to $z$ cross each other) or two consecutive crossings (then
some edge is crossed more than once).  

Let $x$ be a crossing in $\Gamma$, say between edges $(v_0,v_2)$ and $(v_1,v_3)$.  
We call $v_0,v_1,v_2,v_3$ the \emph{endpoints} of the crossing; in a good drawing
without loops these are four distinct vertices.
For $i=0,1,2,3$, we call an edge $e=(v_i,v_{i+1})$ (addition modulo~4)
a \emph{kite-edge of $x$} if the face $F$ of $\GammaP$ incident to $(v_i,x)$ and $(x,v_{i+1})$
is also incident to $e$.
Face $F$ is permitted to be the unbounded face.
We say that $\Gamma$ has \emph{all possible kite-edges} if every crossing has four kite-edges.
Since a missing kite-edge means a cell of degree 4 or more, we have:

\begin{observation}
\label{obs:1ptriangulated}
\label{obs:kite}
In a triangulated 1-planar drawing   all possible kite-edges exist.
\end{observation}

A 1-planar simple graph with $n\geq 3$ has at most $4n-8$ edges \cite{BSW83}
and a similar proof shows that it has at most $4n-8$ cells in any 1-planar drawing.
We strengthen this bound by giving a weighted version.

\begin{lemma}
\label{lem:w0_total}
Let $\Gamma$ be a 1-planar drawing with all possible kite-edges.  For cell $L$
set $w_0(L){:=}1$ if $L$ is crossed, and $w_0(L):=2(\deg(L){-}2)$ otherwise.
Then $\sum_{L\in \Gamma} w_0(L) = 4n-8$.
\end{lemma}
\begin{proof}
Create another drawing $\Gamma'$ by removing,
for every crossing $x$, one of the two crossed
edges.  Since all four kite-edges of $x$ exists, this replaces four crossed cells (with total weight $4\cdot 1$) by
two uncrossed deg-3 cells (with total weight $2\cdot 2$), so the overall weight stays the same.
The resulting drawing $\Gamma'$ is planar, and any face $F$ of $\Gamma'$ has weight $2(\deg(F){-}2)$.
By Lemma~\ref{lem:fd} therefore $\sum_{L\in \Gamma} w_0(L)=\sum_{F\in \Gamma'} 2(\deg(F){-}2) = 4n-8$.
\end{proof}

\paragraph{Maximal 1-planar drawings and graphs:}
We say that a 1-planar drawing $\Gamma$ is \emph{simple-saturated} if it represents a simple
graph, and adding any edge to $\Gamma$ destroy simplicity or 1-planarity.  
In contrast to planar graphs, a simple-saturated  1-planar drawing need not
be triangulated; for example see Figure~\ref{fig:double_K6_1}. 
Vice versa, a triangulated 1-planar drawing need not be saturated 
(it can violate condition \ref{it:twoOnAdjacentCells} defined below), so in contrast to planar
drawings, there is no relationship between `simple-saturated' and `triangulated'.
For future reference we note some properties of a simple-saturated 1-planar drawing
(named $(S_{\dots})$ because the `$S$' reminds of `saturated'); Figure~\ref{fig:conditions} illustrates them.

\begin{figure}[ht]
\hspace*{\fill}
\subcaptionbox{\label{fig:deg5}}{\includegraphics[scale=0.7,page=1,trim=0 0 0 0,clip]{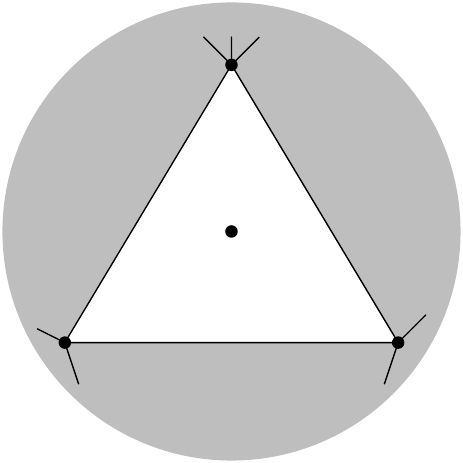}}
\hspace*{\fill}
\subcaptionbox{\label{fig:double_K6_1}}{\includegraphics[scale=1,page=2,trim=0 0 0 0,clip]{double_K6.pdf}}
\hspace*{\fill}
\subcaptionbox{\label{fig:double_K6_2}}{\includegraphics[scale=1,page=3,trim=0 0 0 0,clip]{double_K6.pdf}}
\hspace*{\fill}
\caption{(a) A face with disconnected boundary and degree 5. 
(b) A graph (solid) with a simple-saturated 1-planar drawing that can be made
triangulated only by adding a parallel edge (dashed).
(c) Drawing the same graph differently allows to add a new edge (dashed).
}
\label{fig:properties}
\end{figure}

\begin{observation}
\label{obs:simpleSaturated}
A simple 1-planar drawing $\Gamma$ is simple-saturated if and only if 
\begin{enumerate}[label=($S_{\arabic*}$)]
\vspace*{-2mm}
\itemsep -3pt
\item for any cell that is incident two vertices $z_0\neq z_1$, edge $(z_0,z_1)$ exists, and
	\label{it:twoOnCell}
\item for any uncrossed edge $(u,v)$ with two incident cells $L_0,L_1$, if $L_i$ contains a vertex
	$z_i\neq u,v$ for $i=0,1$, then either $z_0=z_1$ or $(z_0,z_1)$ exists.
	\label{it:twoOnAdjacentCells} 
\end{enumerate}
\vspace*{-2mm}
Furthermore, in any simple-saturated 1-planar drawing
\begin{enumerate}[label=($S_{\arabic*}$)]
\addtocounter{enumi}{2}
\itemsep -3pt
\vspace*{-2mm}
\item no cell-boundary is disconnected, and
	\label{it:disconnected}
\item no cell has multiple incidences with a vertex.
	\label{it:twoIncidences}
\end{enumerate}
\end{observation}
\begin{proof}
If \ref{it:twoOnCell} or \ref{it:twoOnAdjacentCells} were violated, then we could add a new edge $(z_0,z_1)$.
Vice versa, if we can add a new edge $e=(z_0,z_1)$ to $\Gamma$, then \ref{it:twoOnCell} holds (at the cell where
$e$ is inserted) if $e$ is uncrossed and \ref{it:twoOnAdjacentCells} holds (at the edge $(u,v)$ crossed by $e$)
if $e$ is crossed.

If some cell $L$ had a disconnected boundary, then we could find vertices $z_0,z_1$ on each of its circuits,
contradicting \ref{it:twoOnCell} since $z_0,z_1$ could be separated by a closed curve drawn within $L$.   
If $L$ were incident to a vertex $v$ repeatedly,
then each part of the cell-boundary between two occurrences of $v$ contains a vertex, so we can find
two vertices $z_0,z_1$ on $L$, 
contradicting \ref{it:twoOnCell} since $z_0,z_1$ could be separated by a loop at $v$ drawn within $L$.   
\end{proof}

\begin{figure}[ht]
\hspace*{\fill}
\subcaptionbox{\ref{it:twoOnCell}}{\includegraphics[scale=0.7,page=5,trim=0 0 0 0,clip]{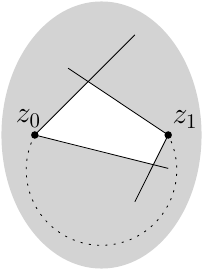}}
\hspace*{\fill}
\subcaptionbox{\ref{it:twoOnAdjacentCells}}{\includegraphics[scale=0.7,page=6,trim=0 0 0 0,clip]{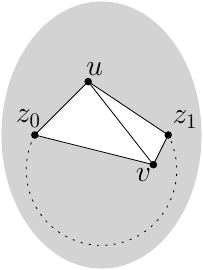}}
\hspace*{\fill}
\subcaptionbox{\ref{it:disconnected}}{\includegraphics[scale=0.7,page=10,trim=0 0 0 0,clip]{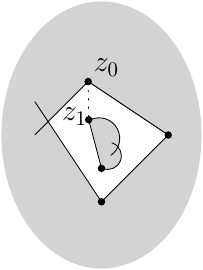}}
\hspace*{\fill}
\subcaptionbox{\ref{it:twoIncidences}}{\includegraphics[scale=0.7,page=7,trim=0 0 0 0,clip]{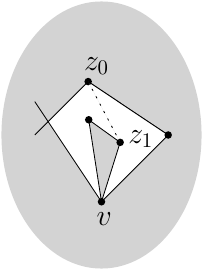}}
\hspace*{\fill}
\subcaptionbox{\ref{it:fromSimple}}{\includegraphics[scale=0.7,page=9,trim=0 0 0 0,clip]{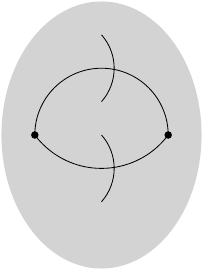}}
\hspace*{\fill}
\subcaptionbox{\ref{it:multiedgeUncrossed}}{\includegraphics[scale=0.7,page=4,trim=0 0 0 0,clip]{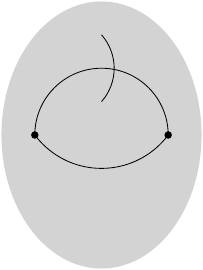}}
\hspace*{\fill}
\caption{Violations to some of our properties.
Dotted edges do not exist (but should).}
\label{fig:properties2}
\label{fig:conditions}
\end{figure}

In a triangulated drawing \ref{it:twoOnCell} (and therefore also \ref{it:disconnected} and \ref{it:twoIncidences})
hold automatically since a cell with three corners has no non-adjacent vertices.

A 1-planar graph $G$ is called \emph{simple-maximal} if it is simple and
{\em every} 1-planar drawing of $G$ is simple-saturated.    Note that this
a stronger condition than `having a simple-saturated 1-planar drawing'; 
for example, the graph in Figure~\ref{fig:double_K6_1} has the simple-saturated
1-planar drawing shown there, but redrawing it as  in Figure~\ref{fig:double_K6_2}
allows to add an edge, so the graph is not simple-maximal.
As opposed to planar graphs, simple-maximal 1-planar graphs do not necessarily have a unique
1-planar drawing, even if they 
are 3-connected (Figure~\ref{fig:not_unique_3conn} will give a specific example).

\section{Existence of large matchings}
\label{sec:lower}

In this section, we prove the lower bounds on $\mu(G)$ listed in Table~\ref{ta:results}.
So we assume that we are given a 1-planar graph $G$ that is either simple-maximal
or that comes with a simple-saturated drawing $\Gamma_0$.

We first give an outline. As with many
previous matching-bound papers (see e.g.~\cite{BDD+04,Dill90,NB79}), we
use the Tutte-Berge-formula  (Theorem~\ref{thm:TutteBerge}).
To lower-bound the size of a maximum matching, it hence suffices to
find an upper bound on $\odd(G{\setminus} S){-}|S|$
for an arbitrary vertex-set $S$, which in turn can be done
by finding an upper bound on $\comp(G{\setminus} S)-|S|$.

To do so, we follow the same idea as used in \cite{BDD+04} to show that
a simple-maximal planar graph $G$ has a matching of size $\tfrac{n+8}{3}$.  
We sketch it briefly here.  Take a planar drawing $\Gamma$ of $G$
and fix an arbitrary vertex set $S$.   
Consider the induced drawing $\Gamma[S]$.  This is a planar drawing (possibly disconnected)
and some of its faces \emph{cover} vertices of $V{\setminus} S$ in the sense that the vertex
belongs to the region of $\mathbb{R}^2$ that defines the face.
Let $f_d^\odot$ be the number of faces of $\Gamma[S]$ that cover vertices of $V{\setminus} S$ 
and that have degree $d$. 
Since the covered vertices form one component of $G{\setminus}S$ \cite{Dill90}, 
we have $\comp(G{\setminus} S)=\sum_d f_d^\odot$.   Combining this with
$\sum_d (d{-}2)f_d = 2|S|{-}4$
yields $\comp(G{\setminus} S)-|S|\leq \tfrac{1}{2}f_3^\odot-2$.
Furthermore, $f_3^\odot \leq \tfrac{1}{3}(2n-4)$ since every face of $\Gamma[S]$ that covers
a vertex of $V{\setminus} S$ covers at least three faces of $\Gamma$.
Together the inequalities give $\odd(G{\setminus}S)-|S|\leq \comp(G{\setminus} S)-|S|\leq \tfrac{n-8}{3}$ and
the matching-bound follows from Theorem~\ref{thm:TutteBerge}.

The argument for a 1-planar graph $G$ principally follows the
same approach, but various steps are much more complicated.  
\begin{itemize}
\itemsep -1pt
\item Fix a 1-planar drawing $\Gamma_0$ of $G$ if not given to us.
	In contrast to planar graphs, $\Gamma_0$ need not be triangulated.  We show
	(in Section~\ref{sec:triangulating}) how to turn $\Gamma_0$ into a
	triangulated drawing $\Gamma$ by
	adding parallel edges while maintaining some properties.
\item For a given vertex-set $S$, the sub-drawing $\Gamma[S]$ is not necessarily
	planar, and we hence cannot talk about its faces.  
	Instead we define (in Section~\ref{sec:patches}) 
	a concept called ``patches'', which mostly correspond
	to uncrossed cells of $\Gamma[S]$ (or actually of a sub-drawing $\Gamma_S$ of $\Gamma[S]$),
	but some patches correspond to crossings of $\Gamma_S$.
\item Since we added parallel edges to $\Gamma$, sub-drawing $\Gamma_S$ may have bigons.
	With this, the upper bound on
	$\comp(G{\setminus} S)-|S|$ (which was $\tfrac{1}{2}f_3^\odot-2$ for planar graphs)
	becomes $|\calP_2^\odot|+\tfrac{1}{2}|\calP_3^\odot|-2$, where $\calP_d^\odot$
	are the patches of degree $d$ that have vertices inside.    But this turns out to be not strong
	enough, and we prove (see Lemma~\ref{lem:towards} in Section~\ref{sec:towards}) 
	a stronger upper bound that considers more types of patches.
\item We need the equivalent of the bound $f_3^\odot\leq \tfrac{1}{3}(2n-4)$.
	One can easily get bounds for $|\calP_d^\odot|$ using weight-function $w_0(\cdot)$
	and Lemma~\ref{lem:w0_total}, but they are not strong enough, and we therefore 
	transfer some weight between cells of $\Gamma$ to get a new weight-function $w_\alpha(\cdot)$
	that has no more total weight (Section~\ref{sec:weights}).
	With an extensive case analysis we then get better lower bounds on the weights of small patches.
\end{itemize}

We put everything together in Section~\ref{sec:together} to prove 
our lower bounds on $\mu(G)$.

\subsection{Triangulating the graph}
\label{sec:triangulating}

As a first step, we want a triangulated 1-planar drawing $\Gamma$ of $G$
(if $G$ comes with a fixed drawing $\Gamma_0$, then $\Gamma$ should be a super-drawing of $\Gamma_0$). 
It is easy to make a 1-planar drawing triangulated by inserting edges, but since we require
some properties of the resulting drawing
(named ($N_{\dots}$) since they relate to non-simplicity), we review the proof here.
We first state the conditions (see also Figure~\ref{fig:conditions}):

\begin{definition}
Define the following properties of a 1-planar drawing:
\begin{enumerate}[label=($N_{\arabic*}$)]
\itemsep -3pt
\item There are no loops.  \label{it:loop}
\item If there are two or more copies of an edge,
	then at most one copy is crossed.
	\label{it:fromSimple}
\item If there are two or more copies of an edge, 
	then all copies are uncrossed.
	\label{it:multiedgeUncrossed}
\end{enumerate}
\end{definition}

Of course \ref{it:multiedgeUncrossed} implies \ref{it:fromSimple}, but we can guarantee
\ref{it:multiedgeUncrossed} only when we can choose $\Gamma_0$.

\begin{lemma}
\label{lem:make_triangulated_drawing}
Let $\Gamma_0$ be a simple-saturated 1-planar drawing with $n\geq 3$ vertices.   Then there exists a triangulated
1-planar drawing $\Gamma$ that is obtained from $\Gamma_0$ by adding uncrossed parallel edges and
that satisfies \ref{it:loop}-\ref{it:fromSimple} as well as \ref{it:twoOnAdjacentCells}.
\end{lemma}
\begin{proof}
Initially set $\Gamma:=\Gamma_0$.     This satisfies \ref{it:loop}-\ref{it:fromSimple} since $\Gamma_0$ is simple.
It likewise has no bigon, and it satisfies \ref{it:twoOnCell}-\ref{it:twoIncidences} since $\Gamma_0$ is simple-saturated;
we will maintain all these properties while adding edges.
Now assume that $\Gamma$ is not yet triangulated.  Since there is no bigon and $n\geq 3$, 
some cell $L$ of $\Gamma$ must have degree 4 or more.  By \ref{it:disconnected}-\ref{it:twoIncidences},
the boundary of $L$ is a single circuit without repeating vertices.   By $\deg(L)\geq 4$, it must
contain two non-consecutive vertices $z_0,z_1$ since there are no consecutive crossings.
By~\ref{it:twoOnCell} an edge $(z_0,z_1)$ exists in $\Gamma$.
Let $\Gamma'$ be the drawing obtained from $\Gamma$ by inserting a new copy of $(z_0,z_1)$
inside $L$. The new edge $e$ has no crossing and  connects two distinct non-consecutive vertices, so
\ref{it:loop}-\ref{it:fromSimple} hold and there is no bigon.
Properties \ref{it:twoOnCell} and \ref{it:twoIncidences} cannot become violated by adding an edge.
Property \ref{it:twoOnAdjacentCells} held for $\Gamma$, so could be violated in $\Gamma'$ only at $e$,
but then \ref{it:twoOnCell} would have been violated at $L$, impossible.   

Now repeat with $\Gamma:=\Gamma'$, which has more cells than $\Gamma$.   Since $\sum_L w_0(L)=4n-8$
and every cell has positive weight (since it has degree 3 or more), there are at most $4n-8$
cells.  So the process will stop eventually with a triangulated drawing.
\end{proof}

If we can choose the initial drawing then we can also achieve \ref{it:multiedgeUncrossed}.

\begin{lemma}
\label{lem:make_triangulated}
Let $G$ be a simple-maximal 1-planar graph.
Then there exists a super-graph $G^+$ of $G$ (obtained by adding parallel edges) and a
triangulated 1-planar drawing $\Gamma$ of $G^+$ 
that satisfies \ref{it:loop}-\ref{it:multiedgeUncrossed} as well as \ref{it:twoOnAdjacentCells}.
\end{lemma}
\begin{proof}
Fix a 1-planar drawing $\Gamma_0$ of $G$ that minimizes the number of crossed edges, and
let $\Gamma$ be the 1-planar drawing obtained from $\Gamma_0$ with Lemma~\ref{lem:make_triangulated}.
We only have to argue that \ref{it:multiedgeUncrossed} holds.   It can be violated only
at an edge $e$ of $\Gamma\setminus \Gamma_0$; all other edges have no parallel copies.
Edge $e$ is uncrossed in $\Gamma$.   If a crossed copy $e'$ of $e$ exists in $\Gamma$,
then $e'$ also existed in $\Gamma_0$ since we only add uncrossed edges.   But then
we could obtain a drawing of $G$ with fewer crossed edges by taking $\Gamma_0$,
removing $e'$ and adding $e$.   Contradiction.
\end{proof}

Since we we have created the triangulated drawing $\Gamma$
by adding parallel edges, we have $\mu(\Gamma)=\mu(G)$, so it suffices to bound $\mu(\Gamma)$.

\subsection{$\Gamma_S$, patches and components they cover}
\label{sec:patches}

So from now on, we will work with a triangulated 1-planar drawing $\Gamma$ that satisfies \ref{it:twoOnAdjacentCells},\ref{it:loop},\allowbreak\ref{it:fromSimple}
and (perhaps) \ref{it:multiedgeUncrossed}.
However, many of our claims below hold even without some of these properties (e.g.~up to Lemma~\ref{lem:patches}
we can have loops), and so we explicitly list the required properties with each claim.

Fix an arbitrary vertex-set $S$;    we want to bound $\comp(\Gamma{\setminus}S)-|S|$.
We will often assume that $\comp(\Gamma{\setminus}S)\geq 2$ ;
the case $\comp(\Gamma{\setminus}S)\leq 1$ is easily handled separately when
we put everything together in Section~\ref{sec:together}.
We need to define a sub-drawing of $\Gamma$ implied by $S$ that will be crucial later.

\begin{definition}
\label{def:GammaS}
Given a triangulated 1-planar drawing $\Gamma$ and a vertex-set $S$, define
sub-drawing $\Gamma_S$ as follows.    
\begin{enumerate}
\itemsep -2pt
\vspace*{-1mm}
\item Begin with the sub-drawing $\Gamma[S]$ induced by the vertices of $S$.  
\item Delete any edge that is crossed by some edge $(u,v)$ in $\Gamma$ 
\label{it:crossing}
	for which $u\not\in S$ or $v\not\in S$.
	See e.g.~edge $(a,b)$ in Figure~\ref{fig:patches}.
\item Delete any uncrossed edge that is incident twice to the same cell of the drawing.
\label{it:cell_simple}
	See e.g.~edge $(c,d)$ in Figure~\ref{fig:patches}.
\end{enumerate}
\end{definition}

\begin{figure}[ht]
\hspace*{\fill}
\subcaptionbox{~}{\includegraphics[scale=0.65,page=1]{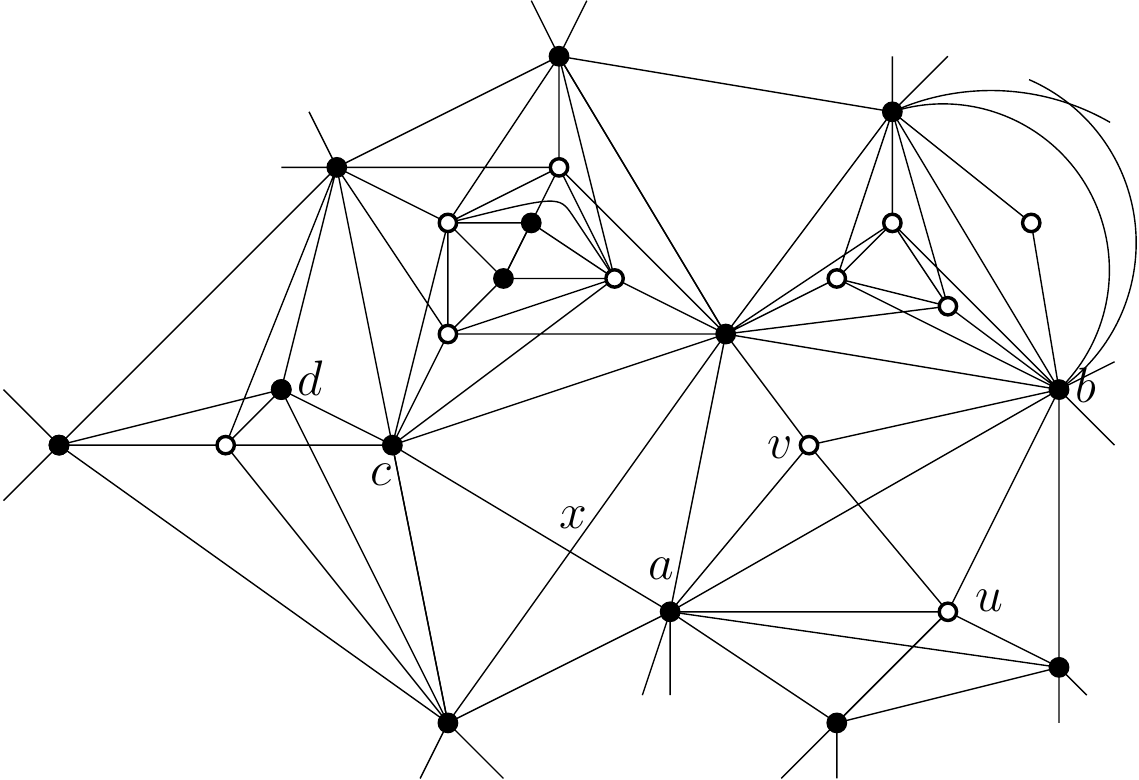}}
\hspace*{\fill}
\subcaptionbox{\label{fig:Gamma_S}}{\includegraphics[scale=0.65,page=2]{patches.pdf}}
\hspace*{\fill}
\caption{(a) Part of a 1-planar drawing. 
Vertices in $S$ are black.
(b) Graph $\Gamma_S$, with patches bold. 
Edges $(a,b)$ and $(c,d)$ belong to $\Gamma[S]$ but not to $\Gamma_S$.
The face-patch $P$ labelled $\calP_8$  has $\comp(P)=2$ (one of its circuits visits a single edge twice), so $\deg(P)=8$.
}
\label{fig:patches}
\end{figure}

Figure~\ref{fig:Gamma_S} illustrates drawing $\Gamma_S$.  
Since all vertices of $\Gamma_S$ are in $S$, any crossing $x$ of $\Gamma_S$
is a \emph{pure-$S$-crossing}: all its endpoints are in $S$.   
Also, $x$ has all four kite-edges in $\Gamma_S$ since those existed in $\Gamma$
by Observation~\ref{obs:kite}.
Let the {\em crossing-patch} $P_x$ be the 4-cycle defined by the four kite-edges of $x$ (these
are uncrossed), and let
the {\em cells covered by $P_x$} be the four cells of $\Gamma$ incident to $x$.   
Let $\calP^\boxtimes$ be the set of all crossing-patches.

Any cell $L$ of $\Gamma_S$ that is not incident to a pure-S-crossing is uncrossed. 
Let the {\em face-patch} $P_L$ of $L$
be the cell-boundary of $L$, i.e., a collection of circuits. 
All edges of $P_L$ are uncrossed in $\Gamma_S$ since $L$ is uncrossed,
and they are uncrossed in $\Gamma$ as well 
due to Step~\ref{it:crossing} of Definition~\ref{def:GammaS}.   
No edge is visited twice by  $P_L$ 
due to Step~\ref{it:cell_simple} of Definition~\ref{def:GammaS}.   
Let the {\em cells covered by $P_L$} be
all those cells of $\Gamma$ that are subsets of cell $L$.
A face-patch $P_L$ is by definition a collection of circuits, but since it
corresponds to a cell $L$ of $\Gamma$ we transfer expressions such as \emph{degree}
and \emph{bigon} and $\comp(\cdot)$ from cell $L$ to patch $P_L$.
See also Figure~\ref{fig:Gamma_S}.
Let $\calP_d$ be the set of face-patches of degree $d$.

Note that any cell of $\Gamma$ is covered by exactly one patch of $\Gamma_S$.
Extending the concept of ``covering'', we say that a patch $P$ {\em covers a
vertex} $v\in V{\setminus}S$ if $v$ is incident to a cell covered by $P$;
this is possible only if $P$ is a face-patch and $v$ lies in region of $\mathbb{R}^2$
that defined the cell of $\Gamma_S$ that defined $P$.
A cell $P$ {\em covers a crossing $x$} of $\Gamma$ if $x$ is incident to a cell covered
by $P$.  (Since edges of any patch are uncrossed, all four cells at $x$ are covered by
the same patch.) Now we want to extend the concept of ``covering'' even
further to components of $\Gamma{\setminus}S$, and must argue that this is well-defined.

\begin{lemma}
\label{lem:isCovered}
Let $\Gamma$ be a triangulated 1-planar drawing and let $S$ be a vertex set.   Then
for every component $C$ of $\Gamma {\setminus} S$, all vertices of $C$ are covered by the same
face-patch of $\Gamma_S$.  
\end{lemma}
\begin{proof}
Assume for contradiction that two different face-patches $P_v$ and $P_w$ cover
vertices $v$ and $w$ of $C$, respectively, and let $L_v$ and $L_w$ be the cells of $\Gamma_S$
that defined $P_v$ and $P_w$.
Let $\pi$ be a path from $v$ to $w$ within $C$, hence not using vertices of $S$.
Path $\pi$ begins inside cell $L_v$ and ends outside it, so it must go through the
boundary of $L_v$.   But this is impossible: vertices of $P_v$ are in $S$ (hence not on $\pi$)
and edges of $P_v$ are uncrossed in $\Gamma$.
\end{proof}

We say that a patch $P$
{\em covers a component} $C$ of $\Gamma {\setminus} S$ if all
vertices of $C$ are covered by~$P$.

\begin{lemma}
\label{lem:atMostOneCovered}
Let $\Gamma$ be a triangulated 1-planar drawing and let $S$ be a vertex set.   Then
every face-patch of $\Gamma_S$ covers at most one component of $\Gamma {\setminus} S$.
\end{lemma}
\begin{proof}
Dillencourt showed an equivalent result for planar graphs
\cite{Dill90}, and the idea is to use the planarization to transfer it
to 1-planar graphs.  So let $C_1,\dots,C_k$ be the components of 
$\Gamma{\setminus} S$ that are covered by a face-patch $P$, and assume for contradiction that $k\geq 2$. 

Consider the planarization $\GammaP$ of $\Gamma$, 
which is a triangulated planar drawing, and let $\GammaP_S$ be the sub-drawing
of $\GammaP$ that corresponds to $\Gamma_S$.  Face-patch $P$ 
corresponds to an uncrossed cell of $\Gamma_S$, hence to a face $F$ of $\GammaP_S$.
Since every vertex in $C_i$ (for $i=1,\dots,k$) also exists in $\GammaP$,
it belongs to some component $C_i^\times$ of $\GammaP\setminus S$ that is covered
by $F$.  Since $\GammaP$ is planar triangulated, any face of $\GammaP{\setminus}S$ covers at most
one component \cite{Dill90}, 
so $C_1^\times=\dots=C_k^\times=:C^\times$.

This does not quite imply that $k=1$, because $C^\times$ also
has dummy-vertices, which could create connections in $\GammaP$ that
did not exist in $\Gamma$.  But it means that any two vertices in
$C_1\cup \dots \cup C_k$ can be connected via a path within $C^\times$.
Let $\pi$ be a shortest path in $C^\times$ that connects a
vertex $v_i\in C_i$ to some vertex $v_j$ in $C_j$ for $1\leq i\neq j\leq k$.
Since $\pi$ is shortest, it can use only dummy-vertices and $v_i,v_j$.  Since
no two dummy-vertices are adjacent in a 1-planar drawing, hence $\pi=\langle v_i,v_j\rangle$ or
$\pi=\langle v_i,d,v_j\rangle$ for some dummy-vertex $d$ of a
crossing $x$ of $\Gamma$.  The former case implies an edge $(v_i,v_j)$.
In the latter case, all kite-edges at $x$ exist, so again $(v_i,v_j)$
exists (crossed or uncrossed) in $\Gamma$.   This is a contradiction
since $v_i,v_j$ are in different components of $\Gamma{\setminus}S$.
\end{proof}

\subsection{Types of face-patches}
\label{sec:towards}

With Lemma~\ref{lem:atMostOneCovered} we can immediately bound $\comp(\Gamma{\setminus}S)$
by the number of face-patches, but we will need a stronger bound.    
Recall the weight function from Lemma~\ref{lem:w0_total}:
$w_0(L)=1$ if cell $L$ is crossed, and $w_0(L)=2(\deg(L)-2)$ otherwise.
(Since $\Gamma$ is triangulated, this means $w_0(L)=2$ for all uncrossed cells of $\Gamma$.)
We can extend this weight-function to a patch $P\in \calP$ 
by setting $w_0(P)$ to be the sum of weights of all cells covered by $P$, and to 
an entire drawing $\Gamma$ by summing of the weights of all cells in $\Gamma$;
we know $w_0(\Gamma)= 4n-8$ from Lemma~\ref{lem:w0_total}.  

Recall that $\calP_d$ denotes the face-patches of degree $d$.
Let $\calP_d^\odot$ be all those face-patches in $\calP_d$ that cover
at least one vertex (hence exactly one component of $\Gamma{\setminus} S$).
For $d\neq 3$ we have $\calP_d^{\odot}=\calP_d$, otherwise $\Gamma$
would have a cell of degree $d\neq 3$.  But for $d=3$ the distinction
matters: Patches in $\calP_3^\nabla:= \calP_3 {\setminus} \calP_3^\odot$
correspond to deg-3 cells of $\Gamma_S$ that are also cells
of $\Gamma$ and do not cover a component, while patches in $\calP_3^\odot$ 
are deg-3 cells of $\Gamma_S$ that are not cells in $\Gamma$ and that
cover a component.  
See also Figure~\ref{fig:Gamma_S}.

\begin{lemma}
\label{lem:patches}
\label{lem:bound}
\label{lem:towards}
Let $\Gamma$ be a triangulated 1-planar drawing. 
Then for any vertex-set $S$ with $\comp(\Gamma{\setminus} S)\geq 2$, we have
$
\comp(\Gamma{\setminus} S)-|S|
\leq \tfrac{1}{2} \big( \sum_{d} (4-d) |\calP^\odot_d|
-2|\calP^\boxtimes| 
-|\calP_3^\nabla|
\big)  - 2.
$
\end{lemma}
\begin{proof}
By Lemma~\ref{lem:isCovered}~and~\ref{lem:atMostOneCovered}, 
every component of $\Gamma{\setminus} S$ is covered by a face-patch
and no face-patch covers two of them. Hence
$\comp(\Gamma{\setminus} S)\leq  \sum_{d} |\calP^\odot_d|.$
Furthermore, by Lemma~\ref{lem:w0_total} we have 
$4|S|-8=\sum_{\text{cell $L$ of $\Gamma_S$}} w_0(L) = 4|\calP^\boxtimes| + \sum_{d} 2(d-2)|\calP_d|$ 
since there are four crossed cells of $\Gamma_S$ in each crossing-patch while an uncrossed cell
$L$ of $\Gamma_S$ has weight $2(\deg(L)-2)$ and corresponds to a face-patch in $\calP_{\deg(L)}$. 
Finally we know $|\calP_d|=|\calP_d^\odot|$ for $d\neq 3$ and
$|\calP_3|=|\calP_3^\odot|+|\calP_3^\nabla|$.
Putting it all together, therefore
\begin{eqnarray*}
4\big( \comp(\Gamma{\setminus} S)-|S| \big) & \leq & 4\sum_{d} |\calP^\odot_d| - (4|S| - 8) - 8 \\
& = & 4\sum_{d} |\calP^\odot_d| - \Big( 4|\calP^\boxtimes| + \sum_{d} 2(d-2)|\calP_d|\Big)  - 8  \\
& = & 4\sum_{d} |\calP^\odot_d| - 4|\calP^\boxtimes| - \sum_{d} 2(d-2)|\calP^\odot_d| - 2|\calP_3^\nabla| - 8  
\end{eqnarray*}
which gives the result after rearranging.
\end{proof}

With this, the problem of upper-bounding $\comp(\Gamma{\setminus}S)-|S|$ becomes the
problem of upper-bounding the number of each type of patch, for which in turn it is enough to
find a lower bound the weight.    For example, since (as we will see)
$w_0(P)\geq 6$ for all $P\in \calP_3^\odot$,  we have $|\calP_3^\odot|\leq \tfrac{1}{6}w_0(\Gamma)=\tfrac{2n-4}{3}$.
We will especially need to bound the weights for patches of small degree, for
which we want to know what such patches can look like.   The following result (illustrated
in Figure~\ref{fig:small_patches}) is straightforward; we give a proof in the appendix for completeness.

\begin{observation}
\label{obs:23}
Assume that $\comp(\Gamma{\setminus} S)\geq 2$ and there are no loops.   Then
\begin{enumerate}
\itemsep -2pt
\item any face-patch of degree 2 is a simple 2-cycle, i.e., a bigon, 
\item any face-patch of degree 3 is a simple 3-cycle,
\item any face-patch of degree 4 is either a simple 4-cycle,
	or a circuit with four distinct edges and three vertices,
	or two circuits, one consisting of two parallel edges and
	the other consisting of a singleton vertex.
\end{enumerate}
\end{observation}

\begin{figure}[ht]
\hspace*{\fill}
\subcaptionbox{~}{\includegraphics[width=0.18\linewidth,page=11]{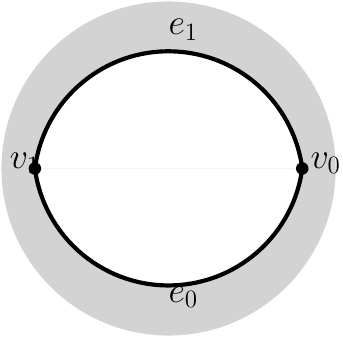}}
\hspace*{\fill}
\subcaptionbox{~}{\includegraphics[width=0.18\linewidth,page=10]{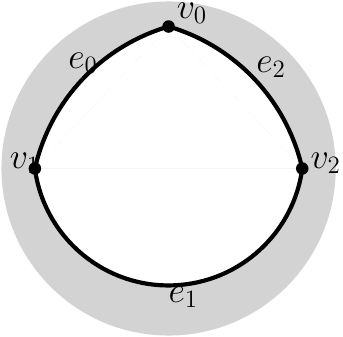}}
\hspace*{\fill}
\subcaptionbox{~}{\includegraphics[width=0.18\linewidth,page=7]{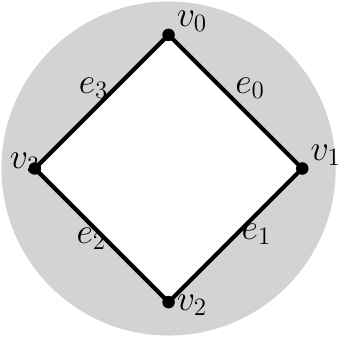}}
\hspace*{\fill}
\subcaptionbox{~}{\includegraphics[width=0.18\linewidth,page=8]{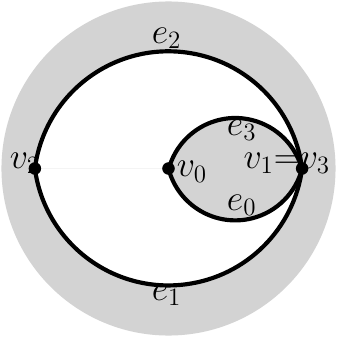}}
\hspace*{\fill}
\subcaptionbox{~}{\includegraphics[width=0.18\linewidth,page=9]{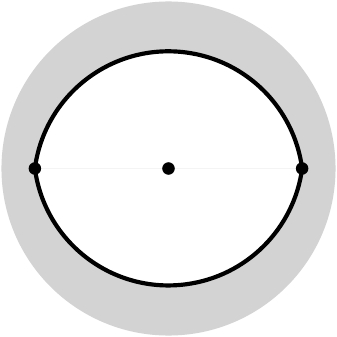}}
\hspace*{\fill}
\caption{Possible configurations a patch (white) of degree 2, 3, 4.}
\label{fig:small_patches}
\end{figure}

We use this first to re-state Lemma~\ref{lem:patches} in a weaker form that nearly always suffices.

\begin{corollary} 
\label{cor:patches}
Let $\Gamma$ be a triangulated 1-planar drawing without loops.   
Then for any vertex-set $S$ with $\comp(\Gamma{\setminus} S)\geq 2$, we have
$
\comp(\Gamma{\setminus} S) -|S| \leq 
|\calP_2| {+} \tfrac{1}{2}|\calP^\odot_3| -\tfrac{1}{2}|\calP^\boxtimes| - \tfrac{1}{4}|\calP_3^\nabla| -2.
$
If $\Gamma$ is 3-connected, then we have
$
\comp(\Gamma{\setminus} S) -|S| \leq 
\tfrac{1}{2}|\calP^\odot_3| -\tfrac{1}{2}|\calP^\boxtimes| - \tfrac{1}{4}|\calP_3^\nabla| -2.
$
\end{corollary}
\begin{proof}
If $\Gamma$ has no loops, then $\calP_1=\emptyset$.   If $\Gamma$ is 3-connected,
then $\calP_2=\emptyset$ since otherwise by $\comp(\Gamma{\setminus}S)\geq 2$ 
there would be vertices both inside and outside the bigon that bounds a patch 
in $\calP_2$, making the two vertices on it a cutting pair.   The bounds
follow by omitting some negative terms from the inequality in Lemma~\ref{lem:patches}.
\end{proof}

Now we bound the weight of some types of patches.

\begin{claim}
\label{cl:w0_small}
Assume that $\comp(\Gamma{\setminus}S)\geq 2$ and there are no loops.
Let $P$ be a face-patch in $\calP^\odot_d$ for $d\in \{2,3,4\}$, and
let $Z(P)$ be the vertices that are covered by $P$.
Then $w_0(P)= 4|Z(P)| + 2(d-2)$.
\end{claim}
\begin{proof}
Let $\Gamma_P$ be the sub-drawing of $\Gamma$ formed by $P$ and all cells
that are covered by $P$.    By Lemma~\ref{lem:w0_total} we have $w_0(\Gamma_P)=4|V(P)|+4|Z(P)|-8$, 
where $V(P)$ denotes the set of vertices that
are on $P$.   This gives a bound on $w_0(P)$ after we subtract the weight
of all cells of $\Gamma_P$ that are not covered by $P$.

Assume first that $P$ is a simple $d$-cycle.
Then the only cell $L$ of $\Gamma_P$ that is not covered by $P$ is the other side
of this $d$-cycle, which has weight $w_0(L)=2(d{-}2)$.   Also $|V(P)|=d$,   so $w_0(P)=4d+4|Z(P)|-8-2(d-2)$
as desired.

If $P$ is not a simple $d$-cycle, then by Observation~\ref{obs:23} we have $d=4$ and $|V(P)|=3$.
Also all cells of $\Gamma_P$ that are not covered by $P$ are bigons that have weight 0, and so
$w_0(P)=w_0(\Gamma)=4 \cdot 3 + 4|Z(P)|-8$ as desired.
\end{proof}

Since any patch in $\calP_d^\odot$ covers at least one vertex, Claim~\ref{cl:w0_small}
implies that $w_0(P)\geq 4$ for $P\in P_2$ and $w_0(P)\geq 6$ for $P\in P_3^\odot$.   Since weights are 
non-negative, we also have $w_0(P)\geq 0$ for any patch.
With this, we can obtain our first lower bound on a matching.

\begin{theorem}
\label{thm:3connDrawing}
Let $\Gamma$ be a 3-connected simple-saturated 1-planar drawing, and $n=8$ or $n\geq 10$.
Then $\Gamma$ has a matching of size at least $\tfrac{n+4}{3}$.
\end{theorem}
\begin{proof}
We may by Lemma~\ref{lem:make_triangulated_drawing} assume that
$\Gamma$ is triangulated and has no loops, for
otherwise we can make it so by adding parallel edges which does not
affect the matching-size.  
Let $M$ be a maximum matching of $\Gamma$; by the Tutte-Berge formula 
$|M| = \tfrac{1}{2}(n-\odd(G{\setminus} S)+|S|)$
for some set $S\subseteq V$.       It hence suffices to show that
$\odd(G{\setminus}S)-|S|\leq \tfrac{n-8}{3}$ for any $S\subseteq V$; 
we will (mostly) upper-bound $\comp(G{\setminus}S)-|S|$ instead.

Assume first that $\comp(\Gamma{\setminus}S)\geq 2$.  Then by Claim~\ref{cl:w0_small}
\begin{eqnarray*}
4n-8 = \sum_{P\in \calP} w_{0}(P) 
& \geq & 4 |\calP_2|  
+ 6 |\calP_3^\odot| 
+ \sum_{\text{patch $P\not\in \calP_2\cup \calP_3^\odot$}} 0
\quad \geq \quad 6|\calP_3^\odot|. 
\end{eqnarray*}
By Corollary~\ref{cor:patches} hence
$
\odd(G{\setminus}S)-|S|\leq \comp(G{\setminus} S)-|S| \leq 
\tfrac{1}{2}|\calP_3^\odot| -2 
\leq  \tfrac{1}{12}(4n{-}8)-2  = \tfrac{n-8}{3}.
$

Now consider the case where $\comp(G{\setminus} S)\leq 1$.
If $n\geq 11$ then $\comp(G{\setminus}S)-|S|\leq 1 \leq \tfrac{n-8}{3}$ holds automatically.
Otherwise $n=8,10$ and $\comp(G{\setminus}S)\leq 1$ implies $\odd(G{\setminus}S)\leq |S|$
since we can (by $n$ even) have an odd component only if $S$ is non-empty.
Therefore $\odd(G{\setminus} S)-|S|\leq 0\leq \tfrac{n-8}{3}$ since $n\geq 8$.
\end{proof}

\subsection{Redistributing weight}
\label{sec:weights}

The bound of Theorem~\ref{thm:3connDrawing} (and in particular the bound $w_0(P)\geq 6$
for $P\in \calP_3^\odot$) is tight, see also Theorem~\ref{thm:3connDrawing_tight}
in Section~\ref{sec:tightness}.
To prove matching-bounds when the drawing is not 3-connected or can be chosen,
we must assign more weight
to patches in  $\calP_2$ and $\calP_3^\odot$ while keeping the overall weight the same.
To define this transfer of weight, we need a few definitions that are illustrated in 
Figure~\ref{fig:transfer_edges}.
Let $E_T$ (the {\em transfer edges})
be all edges that belong to a patch in $\calP_2\cup \calP_3^\odot$. 
Let $\calT$ (the {\em transfer cells}) be all those cells
of $\Gamma$ that are incident to a transfer edge.   
We distinguish
three kinds of transfer cells: 
$\calT^\times$ are transfer cells that are crossed,
$\calT^\nabla$ are transfer cells that are uncrossed and 
use only vertices in $S$ (hence they correspond to patches in $\calP_3^\nabla$),
and $\calT^\circ$
are the remaining cells in $\calT$  (they are uncrossed and
have exactly one vertex not in $S$ since the transfer edge connects vertices of $S$).

\begin{figure}[ht]
\subcaptionbox{~}{
\includegraphics[scale=0.9,page=3,trim=150 80 0 0,clip]{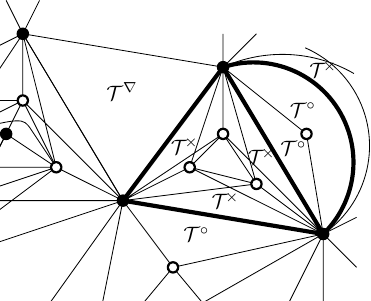}
}
\hspace*{\fill}
\subcaptionbox{\label{fig:nabla_patch_1}}{
\includegraphics[scale=0.9,page=4,trim=0 0 0 0,clip]{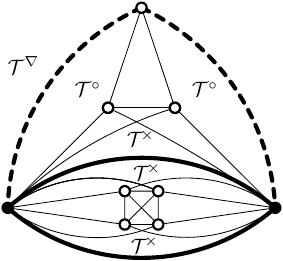}
}
\hspace*{\fill}
\subcaptionbox{\label{fig:nabla_patch_2}}{
\includegraphics[scale=0.9,page=5,trim=0 0 0 0,clip]{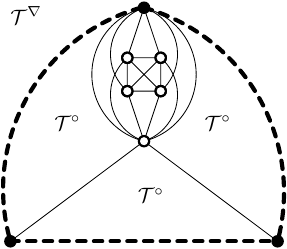}
}
\caption{Transfer edges (bold) and the types of transfer cells. 
Unlabeled cells are not in $\calT$.  (a) A close-up of Figure~\ref{fig:patches}.
(b) A cell in $\calT^\nabla$ can be incident to two transfer edges where the other cell is in $\calT^\circ$ (dashed)
but if (as in (c)) there are three such edges then $\comp(\Gamma{\setminus}S)=1$.}
\label{fig:transfer_edges}
\end{figure}

We define a new weight-function $w_\alpha(\cdot)$ that depends on a parameter $\alpha$ with
$0\leq \alpha\leq 3$.   
Table~\ref{ta:walpha} gives the definition and also the values for $\alpha\in\{\tfrac{4}{3},2,3\}$ that will be
used later.   
To argue that $w_\alpha(\cdot)$ can be seen as shifting weight
across transfer edges, we need some observations.

\begin{table}[ht]\centering
\begin{tabular}{|c||c|c|c|c|c|c||}
\hline
& \multicolumn{3}{|c|}{$L$ transfer cell} 
& \multicolumn{2}{c|}{$L$ not transfer cell}  \\
\cline{2-6}
& $L\in \calT^\times$ & $L\in \calT^\nabla$ & $L\in \calT^\circ$ & $L$ crossed & $L$ uncrossed \\
\hline
$w_\alpha(L):=$ & $w_0(L)-\alpha$ & $w_0(L)-2\alpha$ & $w_0(L)+\alpha$ & $w_0(L)$ & $w_0(L)$ \\
	& $=1-\alpha$ & $=2-2\alpha$ & $=2+\alpha$ & $=1$ & $=2$ \\
\hline
$w_{4/3}(L)$ & $-\tfrac{1}{3}$ & $-\tfrac{2}{3}$ & $\tfrac{10}{3}$ & 1 & 2  \\
\hline
$w_2(L)$ & $-1$ & $-2$ & 4 & 1 & 2 \\
\hline
$w_3(L)$ & $-2$ & $-4$ & 5 & 1 & 2 \\
\hline
\end{tabular}
\caption{Definition of $w_\alpha(L)$ for various types of cell $L$.}
\label{ta:walpha}
\end{table}

\begin{claim}
\label{claim}
\label{cl:notBothUncrossed}
\label{cl:Tcirc}
If $\comp(\Gamma{\setminus}S)\geq 2$ and
\ref{it:loop},\ref{it:twoOnAdjacentCells} hold, then 
for any transfer edge $e$ at least one incident transfer cell
is in $\calT^\times$ or $\calT^\nabla$.
\end{claim}
\begin{proof}
Let $L_0,L_1$ be the two transfer cells incident to $e$, and assume for
contradiction that both are in $\calT^\circ$.  For $i=0,1$, let $z_i$ be 
the vertex of $L_i$ that is not in $S$, see also Figure~\ref{fig:weight_shift_twoCells}.
Let $P\in \calP_2\cup \calP_3^\odot$ be the patch that caused $e$ to be a transfer edge;
by Observation~\ref{obs:23} this is a simple cycle with uncrossed edges.
Since $z_0,z_1\not\in S$, they are not in $P$ and hence separated by it.
This contradicts \ref{it:twoOnAdjacentCells} since neither $z_0\neq z_1$ nor can $(z_0,z_1)$ be an edge.
\end{proof}

By definition of $w_\alpha(\cdot)$, any cell $L\in \calT^\circ$ receives $\alpha$
units of weight (compared to $w_0(\cdot)$).   Since $L$ is a transfer cell,
it must be incident to a transfer edge $e$, and by Claim~\ref{cl:notBothUncrossed}
the other cell $L'$ incident to $e$ is in $\calT^\times$ or $\calT^\nabla$ and
therefore loses at least $\alpha$ units of weight.   But we must argue that
the weight lost by cell $L'$ is not used by too many transfer edges.   This
is trivial for $\calT^\times$, but not at all obvious for $\calT^\nabla$.

\begin{observation}
\label{obs:Ttimes}
If $L'\in \calT^\times$, then at most one incident edge of $L'$ is
a transfer edge that is incident to a cell in $\calT^\circ$.
\end{observation}
\begin{proof}
Cell $L'$ is crossed and has degree 3, so it has only one uncrossed edge
that could be a transfer edge.
\end{proof}

For cells in $\calT^\nabla$, we could have two such incident edges
(see Figure~\ref{fig:nabla_patch_1}), but not three.

\begin{claim}
\label{cl:Tnabla}
If $\comp(\Gamma{\setminus}S)\geq 2$ and
\ref{it:loop},\ref{it:twoOnAdjacentCells} hold, then 
for any transfer cell $L'\in \calT^\nabla$ at most two incident
edges are transfer edges that are incident to a cell in $\calT^\circ$.
\end{claim}
\begin{proof}
Since cell $L'$ corresponds to a patch $P'\in \calP_3^\nabla$, it is
bounded by a simple 3-cycle with three uncrossed edges $e_0,e_1,e_2$
(Observation~\ref{obs:23}).
Figure~\ref{fig:weight_shift_nabla_1},\ref{fig:weight_shift_nabla_2} illustrate this setup, with $L'$ as the
unbounded cell.    
Assume that for all $i\in\{0,1,2\}$ edge $e_i$ is a transfer edge
and its other incident cell $L_i$ belongs to $\calT^\circ$, otherwise we are done.
Write $e_i=(v_i,v_{i+1})$ (addition modulo 3), and let $z_i$
be the vertex of $L_i$ that is not in $S$.  Therefore $z_i\neq v_{i+2}\in S$, and so
by \ref{it:twoOnAdjacentCells} (applied to edge $e_i$), we must have edge $(z_i,v_{i+2})$. 
This in itself is not a contradiction (see the example in Figure~\ref{fig:nabla_patch_2}), but 
it contradicts $\comp(\Gamma{\setminus}S)\geq 2$.

We will show that for some $i\in \{0,1,2\}$, there exists a closed curve $\calC_i$ of
uncrossed edges of $\Gamma$
that separates $z_i$ from $v_{i+2}$; this is a contradiction since then $(z_i,v_{i+2})$ cannot exist.   To
find this curve, let $P_0\in \calP_2\cup \calP_3^\odot$ be the patch 
that caused $e_0$ to be a transfer edge.   Patch $P_0$ must cover cell $L_0$ since cell $L'$
is covered by a patch $P'\in \calP^\nabla$. If $P_0\in \calP_2$, then it is a simple 2-cycle;
this is our desired curve (see Figure~\ref{fig:weight_shift_nabla_1}).   If $P_0\in \calP_3^\odot$, then
it is a simple 3-cycle $\langle v_0,v_{1},s \rangle$ for some 
$s\in S$.   If $s\neq v_2$ then $P$ is our desired curve, so assume that $s=v_2$.
If $P_0$ uses the exact same edges $e_0,e_1,e_2$
that bound $L'$, then $\Gamma_S$ has only two patches: $P_0$ and $P'$.
But then $\comp(\Gamma{\setminus}S)=1$ since $P'\in \calP_3^\nabla$, again a
contradiction.   So we may assume that $P_0$ does not use one of $e_1,e_2$, which means
that there exists a parallel uncrossed copy $(v_i,v_{i+1})$ for some $i\in \{1,2\}$
and curve $\calC_i$ is formed by this edge plus $e_i$.
\end{proof}

\begin{figure}[ht]
\hspace*{\fill}
\subcaptionbox{\label{fig:weight_shift_twoCells}}{\includegraphics[scale=0.9,page=1,trim=0 0 0 0,clip]{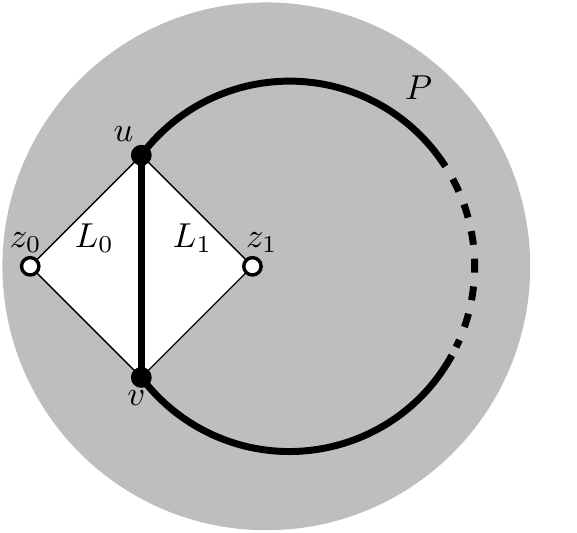}}
\hspace*{\fill}
\subcaptionbox{\label{fig:weight_shift_nabla_1}}{\includegraphics[scale=0.9,page=2,trim=0 0 0 0,clip]{weight_shift.pdf}}
\hspace*{\fill}
\subcaptionbox{\label{fig:weight_shift_nabla_2}}{\includegraphics[scale=0.9,page=3,trim=0 0 0 0,clip]{weight_shift.pdf}}
\caption{(a) Not both cells incident to a transfer edge can be in $\calT^\circ$. 
(b) Two parallel uncrossed edges separate $z_i$ and $v_{i+2}$.
(c) $P_0\in \calP_3^\odot$ implies a parallel uncrossed edge.
}
\label{fig:weight_shift}
\end{figure}  

So there are enough transfer cells in $\calT^\times$ and $\calT^\nabla$ to supply
the weight-gain for $\calT^\circ$.   In fact, some of the weight may get `lost' during
the transfer, but the inequality will work in our favour.   We summarize this
in the following lemma.

\begin{claim}
\label{cl:Lcirc_Ltimes}
If $\comp(\Gamma{\setminus}S)\geq 2$ and
\ref{it:loop},\ref{it:twoOnAdjacentCells} hold, then 
for any $\alpha\geq 0$ we have $w_\alpha(\Gamma)\leq w_0(\Gamma)= 4n-8$.
\end{claim}
\begin{proof}
By 
definition of $w_\alpha(\cdot)$ we
have $w_\alpha(\Gamma)=w_0(\Gamma)+\alpha(|\calT^\circ|{-}|\calT^\times|{-}2|\calT^\nabla|)$.
Write $E_T^\circ\subseteq E_T$ for the transfer edges with an incident cell in $\calT^\circ$.
By Claim~\ref{cl:Tcirc} exactly one incident cell of $e\in E_T^\circ$ is in $\calT^\circ$
and the other is in $\calT^\times\cup \calT^\nabla$.
By Observation~\ref{obs:Ttimes} and Claim~\ref{cl:Tnabla} therefore $|\calT^\circ|
\leq |E_T^\circ| \leq |\calT^\times|+2|\calT^\nabla|$.
By $\alpha\geq 0$ hence $w_\alpha(\Gamma) \leq w_0(\Gamma) = 4n{-}8$.
\end{proof}

\def\mychi{\ensuremath{\raisebox{2pt}{$\chi$}_{\!N_3}}}
The goal is now to give a lower bound on $w_\alpha(P)$ for each type of patch $P$.
Table~\ref{ta:weight_lower} 
summarizes the results, using $\mychi$ to denote the characteristic
function (it is 1 if \ref{it:multiedgeUncrossed} holds and 0 otherwise).  We also 
give the weights for the specific situations that will be encountered later.
Unfortunately, the proofs of these bounds, while
relatively straightforward, are a lengthy case-analysis, 
and we defer them  to Claims~\ref{cl:easy}-\ref{cl:P4} the appendix.

\begin{table}[ht]\centering
\scalebox{0.9}{
\begin{tabular}{|l||c|c|c|c|c|c|}
\hline
$P\in $ & $\calP_2$ & $\calP_3^\odot$ & $\calP_4$ & $\calP_d$  & $\calP^\boxtimes$ & $\calP_3^\nabla$ \\
\hline
\hline
$w_\alpha(P)\geq $ & \multicolumn{1}{|l|}{$\min\{4{+}2\alpha,$} &
	\multicolumn{1}{|l|}{$\min\{6{+}3\alpha,6{+}6\mychi{-}\alpha,$} & 0 &  $\min\{0, (2{-}\alpha)d\}$ & $4{-}4\alpha$ & $2{-}2\alpha$ \\[-0.2ex]
	& \multicolumn{1}{|r|}{~~~$12{+}4\mychi{-}2\alpha\}$} & \multicolumn{1}{|r|}{$10{+}4\mychi{-}3\alpha\}$} & &  & & \\[1.1ex]
Claim & \ref{cl:P2} & \ref{cl:P3} &  \ref{cl:P4} & \ref{cl:Pd} & \ref{cl:easy} & \ref{cl:easy} \\
\hline
\hline
$\alpha = \tfrac{4}{3}$,\ref{it:multiedgeUncrossed} & $\tfrac{20}{3}$ & 10 & 0 & 0 & $-\tfrac{4}{3}$ & $-\tfrac{2}{3}$ \\
\hline
$\alpha = 2$ & 8 & 4 & 0 & 0 & $-4$ & $-2$  \\
\hline
$\alpha = 3$, \ref{it:multiedgeUncrossed} & 10 & 5 & 0 & $-d$ & $-8$ & $-4$  \\
\hline
\end{tabular}
}
\caption{Summary of the lower bounds on $w_\alpha(P)$, assuming $0\leq \alpha\leq 3$
and $\comp(\Gamma{\setminus}S)\geq 2$ as well as \ref{it:loop},\ref{it:fromSimple}.} 
\label{ta:weight_lower}
\end{table}

\subsection{The remaining matching-bounds}
\label{sec:together}

Now we prove the remaining matching bounds, by choosing a suitable parameter
$\alpha$ for each and then combining Claim~\ref{cl:Lcirc_Ltimes}
with the lower bounds on the weights of patches.

\begin{theorem}
\label{thm:not3connDrawing}
Let $\Gamma$ be a simple-saturated 1-planar drawing, and $n=6$ or $n\geq 8$.
Then $\Gamma$ has a matching of size at least $\tfrac{n+6}{4}$.
\end{theorem}
\begin{proof}
We may by Lemma~\ref{lem:make_triangulated_drawing} assume that
$\Gamma$ is triangulated and satisfies \ref{it:loop}-\ref{it:fromSimple},
for otherwise we can make it so by adding parallel edges which does not
affect the matching-size.  ($\Gamma$ does not necessarily satisfy
\ref{it:multiedgeUncrossed}.)
Let $M$ be a maximum matching of $\Gamma$; by the Tutte-Berge formula 
$|M| = \tfrac{1}{2}(n-\odd(G{\setminus} S)+|S|)$
for some set $S\subseteq V$.       It hence suffices to show that
$\odd(S)-|S|\leq \tfrac{n-6}{2}$ for any $S\subseteq V$.

Assume first that $\comp(\Gamma{\setminus}S)\geq 2$.
Using Claim~\ref{cl:Lcirc_Ltimes} and $\alpha=2$ we have by 
Table~\ref{ta:weight_lower} 
\begin{eqnarray*}
4n-8 \geq \sum_{\text{$P$ patch}} w_{2}(P) 
& \geq & 8 |\calP_2|  
+ 4 |\calP_3^\odot| 
+ 0 |\calP_4| 
- \sum_{d\geq 5} 0|\calP_d| 
- 4|\calP^\boxtimes| 
- 2|\calP_3^\nabla|. 
\end{eqnarray*}
By Corollary~\ref{cor:patches} hence
\begin{eqnarray*}
\comp(G{\setminus} S)-|S| & \leq & 
|\calP_2|+\tfrac{1}{2}|\calP_3^\odot| - \tfrac{1}{2}|\calP^\boxtimes| - \tfrac{1}{4}|\calP_3^\nabla| - 2 
\leq \tfrac{1}{8} (4n{-}8) -2 = \tfrac{n-6}{2}.
\end{eqnarray*}

Now assume that $\comp(\Gamma{\setminus}S) \leq 1$.  
If $n\geq 8$ then this implies $\comp(\Gamma{\setminus}S)\leq 1\leq \tfrac{n-6}{2}$ automatically.
Otherwise we have $n=6$, and
$\odd(G{\setminus}S)\leq |S|$ since
we can (by $n$ even) have an odd component only if $S$ is non-empty.
Therefore $\odd(G{\setminus} S)-|S|\leq 0=\tfrac{n-6}{2}$.
\end{proof}

\begin{theorem}
\label{lem:3conn_lower}
\label{thm:3connGraph}
Let $G$ be a 3-connected simple-maximal 1-planar graph with $n\geq 16$ or $n=12,14$.
Then $G$ has a matching of size at least $\tfrac{2n+6}{5}$.
\end{theorem}
\begin{proof}
Using Lemma~\ref{lem:make_triangulated}, we can find a 
1-planar triangulated drawing $\Gamma$ that satisfies \ref{it:loop}-\ref{it:multiedgeUncrossed}
and that represents a graph obtained from $G$ by adding parallel edges, so $\mu(\Gamma)=\mu(G)$.
Let $M$ be a maximum matching of $\Gamma$; by the Tutte-Berge formula 
$|M| = \tfrac{1}{2}(n-\odd(\Gamma{\setminus} S)+|S|)$
for some set $S\subseteq V$.       It hence suffices to show that
that $\odd(\Gamma{\setminus} S)-|S|\leq \tfrac{n-12}{5}$ for any $S\subseteq V$.

Assume first that $\comp(G{\setminus} S)\geq 2$.
Using Claim~\ref{cl:Lcirc_Ltimes} and $\alpha=\tfrac{4}{3}$ we have by 
Table~\ref{ta:weight_lower}
\begin{align*}
4n-8 \geq \sum_{\text{$P$ patch}} w_{4/3}(P) & \geq \tfrac{20}{3} |\calP_2^\odot| +
10 |\calP_3^\odot| - \tfrac{4}{3}|\calP^\boxtimes| - \tfrac{2}{3}|\calP_3^\nabla|  
 \geq 10 |\calP_3^\odot| - 10|\calP^\boxtimes|-5|\calP_3^\nabla|.
\end{align*}
By Corollary~\ref{cor:patches} (and since $G$ and hence also $\Gamma$ is 3-connected) we have
\begin{eqnarray*}
\comp(G{\setminus} S)-|S| & \leq & \tfrac{1}{2}|\calP_3^\odot| - \tfrac{1}{2}|\calP^\boxtimes| - \tfrac{1}{4}|\calP_3^\nabla| - 2 
\leq  \tfrac{1}{20}(4n{-}8)-2 = \tfrac{n-12}{5}.
\end{eqnarray*}

Now consider the case where $\comp(G{\setminus} S)\leq 1$.
If $n\geq 17$ then $\comp(G{\setminus}S)-|S|\leq 1 \leq \tfrac{n-12}{5}$.
If $n=12,14,16$, then $\comp(G{\setminus}S)\leq 1$ implies $\odd(G{\setminus}S)\leq |S|$
since we can (by $n$ even) have an odd component only if $S$ is non-empty.
Therefore $\odd(G{\setminus} S)-|S|\leq 0\leq \tfrac{n-12}{5}$.
\end{proof}

For the fourth bound we need the full power of Lemma~\ref{lem:bound}.
\begin{theorem}
\label{thm:not3connGraph}
Let $G$ be a simple-maximal 1-planar graph, and $n\geq 10$ or $n=8$.
Then $G$ has a matching of size at least $\tfrac{3n+14}{10}$.
\end{theorem}
\begin{proof}
Using Lemma~\ref{lem:make_triangulated}, we can find a 
1-planar triangulated drawing $\Gamma$ that satisfies \ref{it:loop}-\ref{it:multiedgeUncrossed}
and that represents a graph obtained from $G$ by adding parallel edges, so $\mu(\Gamma)=\mu(G)$.
Let $M$ be a maximum matching of $\Gamma$; by the Tutte-Berge formula 
$|M| = \tfrac{1}{2}(n-\odd(\Gamma{\setminus} S)+|S|)$
for some set $S\subseteq V$.       It hence suffices to show that
$\odd(\Gamma{\setminus}S)-|S|\leq \tfrac{2n-14}{5}$ for any $S\subseteq V$.

Assume first that $\comp(\Gamma{\setminus}S)\geq 2$.
Using Claim~\ref{cl:Lcirc_Ltimes} and $\alpha=3$ we have by 
Table~\ref{ta:weight_lower}
\begin{eqnarray*}
4n-8 \geq \sum_{\text{$P$ patch}} w_{3}(P) 
& \geq & 10 |\calP_2|  
+ 5 |\calP_3^\odot| 
+ 0 |\calP_4| 
- \sum_{d\geq 5} d|\calP_d| 
- 8|\calP^\boxtimes| 
- 4|\calP_3^\nabla| \\
& \geq & 10 |\calP_2|  
+ 5 |\calP_3^\odot| 
- \sum_{d\geq 5} (5d-20)|\calP_d| 
- 10|\calP^\boxtimes|
- 5|\calP_3^\nabla| \\
& = & \sum_{d\geq 2} (20-5d)|\calP_d^\odot| 
- 10|\calP^\boxtimes|
- 5|\calP_3^\nabla|
\end{eqnarray*}
since $|\calP_d|=|\calP_d^\odot|$ for $d\neq 3$ and $d\leq 5d-20$ for $d\geq 5$.  
By Lemma~\ref{lem:bound} (and since $|\calP_1|=0$ in a graph without loops) hence
$$\comp(G{\setminus} S){-}|S|
\leq \tfrac{1}{2} \Big( \sum_{d\geq 2} (4{-}d) |\calP^\odot_d|
-2|\calP^\boxtimes| 
-|\calP_3^\nabla|
\Big)  - 2
\leq \tfrac{1}{10} (4n{-}8) -2 = \tfrac{2n-14}{5}.$$
Now assume that $\comp(\Gamma{\setminus}S) \leq 1$.
If $n\geq 10$ then this implies $\comp(\Gamma{\setminus}S)-|S|\leq 1\leq \tfrac{2n-14}{5}$ automatically.
Otherwise we have $n=8$, and
$\odd(G{\setminus}S)\leq |S|$ since
we can (by $n$ even) have an odd component only if $S$ is non-empty.
Therefore $\odd(G{\setminus} S)-|S|\leq 0\leq \tfrac{2n-14}{5}$. 
\end{proof}

All our lower bounds (Theorem~\ref{thm:3connDrawing},\ref{thm:not3connDrawing},\ref{thm:3connGraph},\ref{thm:not3connGraph})
exclude some small values of $n$.   One can easily argue that this is required; we only do this for Theorem~\ref{thm:not3connGraph}
and leave the others to the reader.   If $G$ has a matching $M$ of size at least $\tfrac{3n+14}{10}$, then $|M|\geq \lceil \tfrac{3n+14}{10} \rceil$
by integrality and therefore at least $2\lceil \tfrac{3n+14}{10} \rceil$ vertices are matched.   
If $n=9$ or $n\leq 7$ then $2\lceil \frac{3n+14}{10} \rceil > n$, which means that Theorem~\ref{thm:not3connGraph} cannot possibly hold.

\section{Tightness}
\label{sec:tightness}

In this section, we show that the lower bounds on the matching-size are tight for all
classes of graphs/drawings that we considered.

\subsection{The bipyramid and its 1-planar drawings}

Our constructions are all based on the {\em bipyramid} $B_s$, whose construction depends on a parameter $s\geq 5$
(we will frequently also require $s$ to be even).  It consists of a \emph{base-cycle} $C_B$ of length $s-2$,
and two \emph{apices} $c,c'$ adjacent to all vertices of $C_B$.  See Figure~\ref{fig:bipyramid}.   We call
the edges of $C_B$ the \emph{base-edges} and the other edges the \emph{apex-edges}.
Note that $B_s$ is planar and triangulated
and hence has a unique planar drawing with $2s-4$ faces.  It also has $s$ vertices and $3s-6$ edges. 
If $s$ is even then all its vertex-degrees are even; in consequences the faces of
$B_s$ can be colored black and white with no two adjacent faces of the same color,
and there are $s-2$ black and white faces each.    Define $S$ to be the $s$
vertices of $B_s$; this will be the set that we use to bound $\mu(G)$ via $\comp(G\setminus S)-|S|$
below.   

\def\myfactor{0.99}
\begin{figure}[ht]
\hspace*{\fill}
\subcaptionbox{$B_6$.\label{fig:bipyramid}}{\includegraphics[scale=\myfactor,page=1,trim=15 0 15 0,clip]{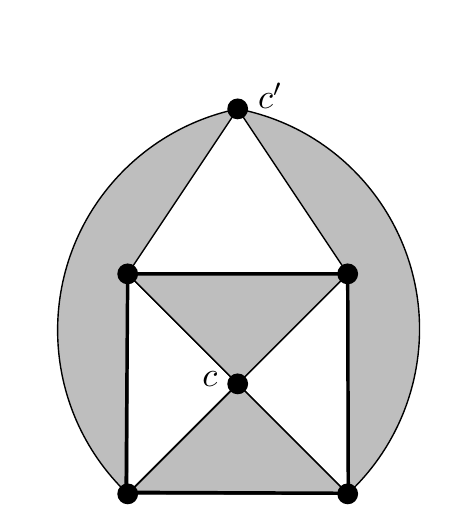}}
\hspace*{\fill}
\subcaptionbox{\label{fig:bipyramid_redrawn}}{\includegraphics[scale=0.8,page=13]{tight.pdf}}
\hspace*{\fill}
\subcaptionbox{$B_8^T$.\label{fig:bipyramid_triangled}}{\includegraphics[scale=\myfactor,page=12]{tight.pdf}}
\hspace*{\fill}
\subcaptionbox{\label{fig:bipyramid_notplanar}}{\includegraphics[scale=\myfactor,page=14]{tight.pdf}}
\hspace*{\fill}
\caption{(a) The bipyramid with its unique planar drawing; base-edges are thick.   (b) A 1-planar drawing of $B_{10}$.  
(c) Attaching triangles. (d) A 1-planar drawing of $B_6^T$ where the induced drawing of $B_6$ is not planar;
this is not a violation to Lemma~\ref{lem:bipyramid_planar} since $s<8$.
}
\end{figure}

1-planar drawings of $B_s$ are (sadly) not always planar, see Figure~\ref{fig:bipyramid_redrawn}.   However, 
by attaching small subgraphs we can force $B_s$ to be drawn planar.
Formally, the operation of \emph{attaching a triangle at an edge $e=(u,v)$} means to add a new
vertex $z_e$ and make it adjacent to both $u$ and $v$ and to no other vertices.

\begin{lemma}
\label{lem:bipyramid_planar_base}
Let $B^T_s$ be the graph obtained from the bipyramid $B_s$ by attaching triangles at all base-edges. If $s\geq 8$
then in any 1-planar drawing $\Gamma$ of $B^T_s$ the base-edges are uncrossed.
\end{lemma}
\begin{proof}
Let $e=(u,v)$ be a base-edge and 
let $T=\langle u,v,z_e\rangle$ be the triangle attached at $e$.
None of the edges $(u,v),(v,z_e),(z_e,u)$ can cross each other since they have common endpoints, so they form a simple
closed curve $\calC_T$.    See also Figure~\ref{fig:bipyramid_triangled}.
Let $n_0$ and $n_1$ be the number of base-cycle vertices that are strictly
outside and inside $\calC_T$; we have $n_0+n_1=s-4\geq 4$ since there are $s-2$ base-cycle vertices, and only $u,v$ are on $\calC_T$.   
If apices $c,c'$ are on opposite sides of $\calC_T$, then 
$n_0+n_1$ apex-edges cross $\calC_T$.   Since the drawing is 1-planar drawing and curve $\calC_T$ consists
of three edges, this implies $n_0+n_1\leq 3$, impossible.
So $c,c'$ are (say) both outside  $\calC_T$ and at least $2n_1$ edges cross $\calC_T$.   

This implies $n_1\leq 1$ by integrality and since at most three edges cross $\calC_T$.   
If $n_1=1$ (say $w$ is inside $\calC_T$), then there are six edge-disjoint paths from $w$
to other vertices of $B_s$: The edges $(w,c)$ and $(w,c')$, the two base-edges incident to $w$, and the 
two paths along the triangles attached at these base-edges.   Since at most one of $u,v$ can be adjacent to
$w$, four of these paths lead from $w$ (inside $\calC_T$) to vertices strictly outside  $\calC_T$.   Again this violates
1-planarity. 
So all vertices of $B_s$ are $u,v$ or strictly outside $\calC_T$.   

Assume now that there is an edge $e'$ that crosses $e=(u,v)$, hence $e'$ crosses $\calC_T$.   So $e'$
reaches points inside $\calC_T$, and cannot cross $\calC_T$ again by 1-planarity, hence must end at one of $u,v,z_e$.   
But it cannot end at $z_e$ since $z_e$ only has neighbours $u,v$. So $e'$ and $e$ share an endpoint, contradicting that we
have a good drawing.
\end{proof}

\begin{lemma}
\label{lem:bipyramid_planar}
Let $B^T_s$ be a graph obtained from the bipyramid $B_s$ by attaching triangles at all base-edges. If $s\geq 8$,
then in any 1-planar drawing $\Gamma$ of $B^T_s$ the induced drawing $\Gamma_B$ of $B_s$ is planar.
\end{lemma}
\begin{proof}
Assume for contradiction that two edges of $B_s$ cross each other at $x$; from the previous lemma we know that neither edge is a base-edge.
They also cannot be adjacent to the same apex-vertex since we have a good drawing, so both $c$ and $c'$ are endpoints of crossing $x$.
We know that the base-cycle $C_B$ is uncrossed; let $\calC_B$ be the corresponding simple closed curve in $\Gamma_B$
with $x$ (say) inside.   Therefore $c,c'$ are also inside $\calC_B$, and with them, 
\emph{all} edges incident to $c,c'$ since they cannot cross $\calC_B$.   So the only vertices possibly outside $\calB$ in $\Gamma$ are 
deg-2 vertices of attached triangles; these are not in $\Gamma_B$.

So the unbounded cell of $\Gamma_B$ is incident to all of $\calC_B$.   We can now obtain a 1-planar drawing of $K_{4,s-2}$ as follows:
Duplicate $\Gamma_B$, invert the plane for the copy and indentify the two copies of $\calC_B$, then delete the base-edges. 
This is impossible since 1-planar bipartite graphs contain at most $3n-8$ edges \cite{Kar13} but $K_{4,s-2}$ has $s+2$ vertices and (by $s\geq 8$)
$4s{-}8>3(s{+}2){-}8$ edges. 
\end{proof}

\subsection{Maximal 1-planar graphs}

To define our constructions, we need two more methods of modifying the bipyramid $B_s$.
Let $\Gamma$ be a simple planar drawing and $F$ be a deg-3 face of $\Gamma$.   
To \emph{insert a $K_4$} in $F$ means to add a new vertex $z_F$ and to make it adjacent to
all vertices of $F$ (so $F\cup \{z_F\}$ forms a $K_4$).
To \emph{insert a $K_6$} in $F$ means similarly to add three new vertices and to make them adjacent to
all vertices of $F$ as well as to each other.

\def\myfactor{0.99}
\begin{figure}[ht]
\hspace*{\fill}
\subcaptionbox{$G_1$.\label{fig:3connGraph_tight}}{\includegraphics[scale=\myfactor,page=3,trim=0 5 10 0,clip]{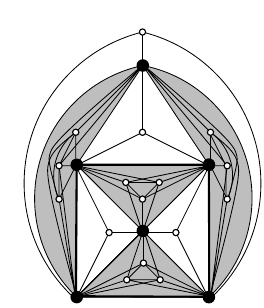}}
\hspace*{\fill}
\subcaptionbox{Alternative drawing.\label{fig:not_unique_3conn}}{\includegraphics[scale=1.2,page=20,trim=0 30 0 0,clip]{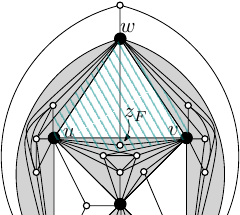}}
\hspace*{\fill}
\subcaptionbox{$G_2$.\label{fig:not3connGraph_tight}}{\includegraphics[scale=1,page=5,trim=0 0 0 0,clip]{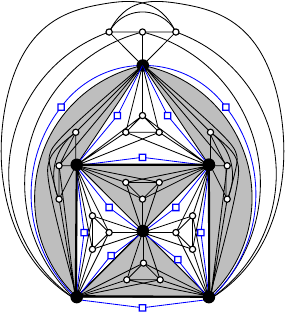}}
\hspace*{\fill}
\caption{Graphs to show that the matching bounds are tight.  
For ease of drawing we show here $s=6$, though the actual constructions use $s\geq 8$.
The striped region in (b) indicates all possible locations for $z_F$. } 
\label{fig:tight}
\end{figure}

\begin{theorem}
\label{thm:3connGraph_tight}
For any $N$, there exists a 3-connected simple-maximal 1-planar graph with 
$n\geq N$ vertices such that any matching has size at most $\tfrac{2n+6}{5}$.
\end{theorem}
\begin{proof}
Set $s=\max\{8,\tfrac{N+8}{5}\}$ rounded up to the nearest even integer. 
Take the unique planar drawing $\Gamma_B$ of graph $B_s$; since $s$ is even we can color its faces
as black and white.   
Define $G_1$ to be the graph obtained from $\Gamma_s$ by inserting $K_4$ into every white face and $K_6$ into every black face.
See Figure~\ref{fig:3connGraph_tight}. 
To see that $G_1$ is 3-connected, recall that $B_s$ is 3-connected, and $G_1$ can be obtained from it by
repeatedly adding vertices with at least three distinct neighbours; this operation maintains 3-connectivity.
There are $s-2$ black and $s-2$ white faces in $\Gamma_B$, so $G_1$ has $n=s+(s-2)+3(s-2)=5s-8\geq N$
vertices.  Also, $G_1{\setminus} S$ has $2s-4$ odd components (one per face of $B_s$), so 
$\odd(G_1{\setminus} S)-|S|=s-4=\tfrac{n-12}{5}$ and by the Tutte-Berge formula therefore
$\mu(G_1)\leq \tfrac{2n+6}{5}$. 

It remains to show that $G_1$ is simple-maximal 1-planar,
i.e., any 1-planar drawing  $\Gamma'$ of $G_1$ is simple-saturated.   Observe that $G_1$ contains $B_s^T$ 
as a subgraph (the inserted $K_4$'s create triangles at the base-edges).   By Lemma~\ref{lem:bipyramid_planar}
the drawing $\Gamma_B'$ of $B_s$ inside $\Gamma'$ is planar
and hence exactly drawing $\Gamma_B$ by uniqueness of planar drawings of $B_s$.

For every black face $F=\{u,v,w\}$ of $\Gamma_B$, we added three new vertices $z,z',z''$ to form a $K_6$.
As argued in \cite{FHM+20}, this forces $z,z',z''$ to lie in face $F$, because 
1-planar drawings of $K_6$ are unique up to renaming, and $\langle u,v,w\rangle$
is drawn uncrossed.

For every white face $F=\{u,v,w\}$ of $\Gamma_B$, we added a new vertex $z_F$, 
adjacent to all of $u,v,w$.    There are only two possible 1-planar drawings of $K_4$ (up
renaming); either it is drawn planar or with exactly one crossing.     So $z_F$ is in
a face of $\Gamma_B$ that contains two of $\{u,v,w\}$, hence either in $F$ or one of
its adjacent black faces.  See also Figure~\ref{fig:not_unique_3conn}.
If $z_F$ is in a black face $F'$, 
then $F'$ also contains a $K_6$; vertex $z_F$ must avoid edges of this $K_6$ (else some edge
would be crossed twice), and so stays in the cell of $\Gamma'\setminus z_F$ immediately adjacent to 
common edge of $F$ and $F'$.  
This describes all possible ways of drawing $\Gamma'$, and as one verifies
\ref{it:twoOnCell} and \ref{it:twoOnAdjacentCells} hold for all possibilities and $\Gamma'$ is simple-saturated.
\end{proof}

\begin{theorem}
\label{thm:general_tight}
\label{thm:not3connGraph_tight}
For any $N$, there exists a simple-maximal 1-planar graph with 
$n\geq N$ vertices such that any matching has size at most $\tfrac{3n+14}{10}$.
\end{theorem}
\begin{proof}
Set $s=\max\{8,\tfrac{N+18}{10}\}$ rounded up to the nearest even integer. 
Define $G_2$ to be the graph obtained from $B_s$ by inserting $K_6$ into every face,
and attaching a triangle at every edge of $B_s$.
See Figure~\ref{fig:not3connGraph_tight}.
There are $2s-4$ faces and $3s-6$ edges in $B_s$, so $G_2$ has $n=s+3(2s-4)+(3s-6)=10s-18\geq N$
vertices.  Also, $G_2{\setminus} S$ has $(2s-4)+(3s-6)=5s-10$ odd components (one per face and
edge of $B_s$), so $\odd(G_2{\setminus} S)-|S|=4s-10=\tfrac{2n-14}{5}$ and $\mu(G_2)\leq \tfrac{3n+14}{10}$.

We must argue that $G_2$ is simple-maximal 1-planar.   
Graph $G_2$ contains a copy of $B_s^T$, so in any 1-planar drawing $\Gamma'$ of $G_2$ the induced drawing of $B_s$ is planar,
hence exactly $\Gamma_B$.   All the inserted $K_6$'s must be inserted in their corresponding faces of
$\Gamma_B$.   Let $\Gamma^+$ be the drawing $\Gamma_B$ with all $K_6$'s inserted, and consider an edge
$e=(u,v)$ of $B_s$;  we attached a triangle with new vertex $z_e$ at $e$.  One verifies that all incident
cells of $u$ in $\Gamma^+$ are crossed.   Therefore $(u,z_e)$ must be drawn uncrossed in $\Gamma'$, otherwise
some edge would be crossed twice.   Likewise $(v,z_e)$ must be drawn uncrossed in $\Gamma'$.   Therefore
$z_e$ is placed in one of the two cells of $\Gamma^+$ incident to $(u,v)$.
This describes all possible ways of drawing $\Gamma'$, and as one verifies
\ref{it:twoOnCell} and \ref{it:twoOnAdjacentCells} hold for all possibilities and $\Gamma'$ is simple-saturated.
\end{proof}

\subsection{Maximal 1-planar drawings}
To construct 1-planar drawings with small matchings, 
we begin again with bipyramid $B_s$, but this time insert $K_4$ with a crossing in each face.   Formally, note that we can create a pairing
between the faces of $B_s$ and the apex-edges such that every face is paired with an incident edge.
See Figure~\ref{fig:pairing}.
To {\em insert $K_4^\times$} into a face $F$ means to add a new vertex $v$ inside $F$, make it adjacent to all three vertices of
$F$, and then re-route the edge $e$ that was paired with $F$ so that $e$ crosses the non-incident edge at $v$.

\begin{theorem}
\label{thm:3connDrawing_tight}
For any $N$, there exists a 3-connected simple-saturated 1-planar drawing $\Gamma_3$ with $n\geq N$ vertices
such that $\mu(\Gamma_3)\leq \tfrac{n+4}{3}$.
\end{theorem}
\begin{proof}
Set $s=\max\{8,\tfrac{N+4}{3}\}$ rounded up to the nearest even integer. 
Let $\Gamma_3$ be the drawing obtained from the planar drawing of $B_s$ by inserting $K_4^\times$ into each face, see also Figure~\ref{fig:3connDrawing_tight}.
One verifies that $\Gamma_3$ is simple-saturated.  It is 3-connected since it has a triangulated planar drawing (replace the crossed
edges by the dotted edges in Figure~\ref{fig:3connDrawing_tight}).
There are $2s-4$ faces in $B_s$, so $\Gamma_3$ has $n=s+(2s-4)=3s-4\geq N$ vertices.
Also, $\Gamma_3 {\setminus} S$ has $2s-4$ odd components (one per face of $B_s$) 
so $\odd(\Gamma_3{\setminus} S)-|S|=s-4=\tfrac{n-8}{3}$ and $\mu(\Gamma_3)\leq \tfrac{n+4}{3}$.
\end{proof}

\begin{theorem}
\label{thm:not3connDrawing_tight}
For any $N$, there exists a simple-saturated 1-planar drawing $\Gamma_4$ with $n\geq N$ vertices
such that $\mu(\Gamma_4)\leq \tfrac{n+6}{4}$.
\end{theorem}
\begin{proof}
Set $s=\max\{8,\tfrac{N+6}{4}\}$ rounded up to the nearest even integer. 
Let $\Gamma_3$ be the drawing of Theorem~\ref{thm:3connDrawing_tight} and attach
a triangle at every base-edge, see also Figure~\ref{fig:not3connDrawing_tight}.
One verifies that the resulting drawing $\Gamma_4$ is simple-saturated.
There are $s-2$ base-edges and $2s-4$ faces in $B_s$, so $\Gamma_3$ has $n=s+(s-2)+(2s-4)=4s-6\geq N$ vertices.
Also, $\Gamma_4 {\setminus} S$ has $(s-2)+(2s-4)$ odd components (one per base-edge and face of $B_s$),
so $\odd(\Gamma_4{\setminus} S)-|S|=2s-6=\tfrac{n-6}{2}$ and $\mu(\Gamma_4)\leq \tfrac{n+6}{4}$.
\end{proof}

\begin{figure}[ht]
\hspace*{\fill}\
\subcaptionbox{The pairing. \label{fig:pairing}}{\includegraphics[scale=\myfactor,page=6]{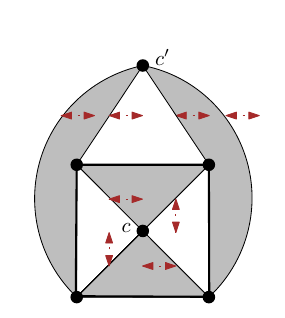}}
\hspace*{\fill}\
\subcaptionbox{$\Gamma_3.$ \label{fig:3connDrawing_tight}}{\includegraphics[scale=\myfactor,page=7]{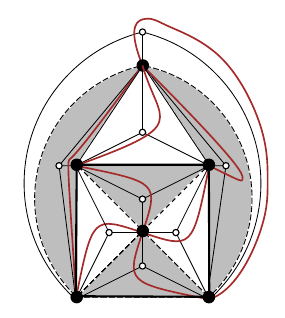}}
\hspace*{\fill}\
\subcaptionbox{$\Gamma_4.$ \label{fig:not3connDrawing_tight}}{\includegraphics[scale=\myfactor,page=8]{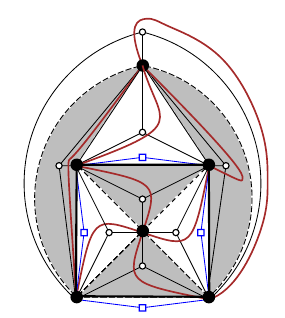}}
\hspace*{\fill}
\caption{Drawings to show that the matching bounds are tight. 
For ease of drawing we show here $s=6$, though the actual constructions use $s\geq 8$.
}
\label{fig:tightDrawing}
\end{figure}

\section{Non-simple graphs and drawings}
\label{sec:non_simple}

In all our results, we assumed that the input graph (or drawing) is simple.
This is usually a reasonable assumption, since adding parallel edges or
loops cannot increase the matching-size.   However, adding parallel edges
and loops \emph{can} make a drawing saturated that previously was not
saturated, so the class of `saturated drawings' or `maximal graphs' can
change, and with it, the size of matchings that always exists for such
drawings/graphs.

Before explaining this further, we first need to clarify what `saturated'
even means for drawings that are not simple.   We cannot permit cells of 
degree one or two, because otherwise we can always add more edges.
Call a cell \emph{proper} if it has degree  3 or more, and 
a 1-planar drawing $\Gamma$ \emph{proper-cell-saturated} if 
all its cells are proper and adding any edge destroys 1-planarity or 
creates a non-proper cell.

Assume first that the input graph is allowed to have loops.   Nearly all our claims assumed that
there are no loops (i.e,  condition~\ref{it:loop}),
so it is no surprise that no good matching-bounds exist in the presence of loops.   To see
this, take $K_{1,n}$, and add loops at the center vertex 
that separate all other vertices from each other.
See Figure~\ref{fig:notSimpleLoops} for the resulting triangulated drawing.  
One verifies \ref{it:twoOnAdjacentCells}, so this is proper-cell saturated,
but at most one edge can be in a matching.
However, if we have no loops, then we can prove non-trivial matching-bounds.

\begin{theorem}
\label{thm:notSimpleDrawing}
Let $\Gamma$ be a 1-planar proper-cell-saturated drawing $\Gamma$ that has no loops. 
Then (for sufficiently large $n$) $\Gamma$ has a matching of size at least 
$\tfrac{n+6}{4}$, and at least $\tfrac{n+4}{3}$ if $\Gamma$ is 3-connected.
\end{theorem}
\begin{proof}
We use essentially the same proof as for Theorem~\ref{thm:3connDrawing} and \ref{thm:not3connDrawing},
and explain here only what is different now.
We again first triangulate $\Gamma$ by adding uncrossed edges as in
Lemma~\ref{lem:make_triangulated_drawing}.   The resulting drawing
satisfies \ref{it:loop} because $\Gamma$ had no loops and we do not add any.
It also satisfies \ref{it:twoOnAdjacentCells}, but it may violate
\ref{it:fromSimple}. 

Inspecting Section~\ref{sec:lower}, we see that \ref{it:fromSimple} is only needed for the results
of Table~\ref{ta:weight_lower}.   The corresponding claims in Appendix~\ref{app:weight_lower}
make clear exactly when \ref{it:fromSimple} is needed:
\begin{enumerate}
\itemsep -2pt
\item for a patch in $\calP_2$ that covers an even number of vertices, and
\item for a patch in $\calP_4$ if $\alpha>2$.
\end{enumerate}
The proofs of Theorem~\ref{thm:3connDrawing} and \ref{thm:not3connDrawing}
use $\alpha=0$ and $\alpha=2$, so the second issue resolves automatically.
The first issue means that we no longer have the
upper bound on $\comp(\Gamma{\setminus}S)$, but the exact same proof works for an
upper bound on $\odd(\Gamma{\setminus}S)$, which is all that is needed for the matching-bound.
\end{proof}

The same drawings as for the simple case (Theorem~\ref{thm:3connDrawing_tight}
and \ref{thm:not3connDrawing_tight}) show that these bounds are tight.
But surprisingly, the bounds are tight even for 
\emph{proper-cell-maximal 1-planar graphs}, i.e., when we can freely
choose the 1-planar drawing as long as all cells are proper.

\begin{figure}[ht]
\hspace*{\fill}
{\subcaptionbox{\label{fig:notSimpleLoops}}{\includegraphics[width=0.12\linewidth,page=3]{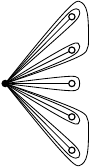}}}
\hspace*{\fill}
\subcaptionbox{$G_5$ (closeup).\label{fig:3connGraphNonSimple_tight}}{\includegraphics[scale=1.2,page=21,trim=15 20 15 50,clip]{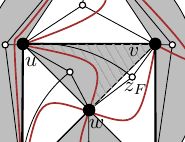}}
\hspace*{\fill}
\subcaptionbox{$G_6$.\label{fig:not3connGraphNonSimple_tight}}{\includegraphics[scale=1,page=10,trim=10 0 8 0,clip]{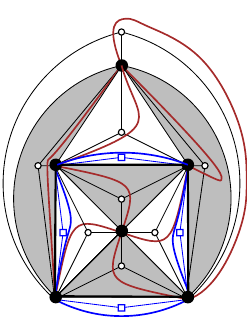}}
\hspace*{\fill}
\subcaptionbox{\label{fig:NonSimpleBadDrawing}}{\includegraphics[scale=1.2,page=15,trim=15 30 15 50,clip]{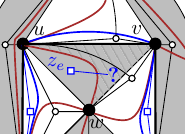}}
\hspace*{\fill}
\caption{Non-simple graphs with proper-cell drawings. (b) Cell $F'$ (striped) is bisected by an edge incident to $z_F$. 
(d) No edge incident to $z_e$ can be crossed.}
\label{fig:notSimple}
\end{figure}

\begin{theorem}
\label{thm:not3connGraphNonSimple_tight}
For any $N$, there exists a 3-connected proper-cell-maximal 1-planar graph without loops and with 
$n\geq N$ vertices such that any matching has size at most $\tfrac{n+4}{3}$.
There also exists a proper-cell-maximal 1-planar graph without loops and with 
$n\geq N$ vertices such that any matching has size at most $\tfrac{n+6}{4}$.
\end{theorem}
\begin{proof}
Take the drawing $\Gamma_3$ from Theorem~\ref{thm:3connDrawing_tight} (i.e., the bipyramid
with $K_4^\times$ inserted in every face) and duplicate the apex-edges, see also Figure~\ref{fig:3connGraphNonSimple_tight}. 
Call the result $G_5$, and observe that it is 3-connected, has
a proper-cell 1-planar drawing, and $\mu(G_5)\leq \tfrac{n+4}{3}$ by Theorem~\ref{thm:3connDrawing_tight}. 
It remains to show any proper-cell 1-planar drawing $\Gamma$ of $G_5$ is proper-cell saturated.   

Let $\Gamma_B^\shortparallel$ be the sub-drawing of $\Gamma$ consisting of all edges of $B_s$, including
parallel apex-edges.  Since $G_5$ contains $B_s^T$, and we could have used either edge of any parallel pair
for $B_s^T$, drawing $\Gamma_B^\shortparallel$ must be planar.   
Since $B_s$ is 3-connected, drawing $\Gamma_B^\shortparallel$ therefore is unique, and each pair of parallel 
apex-edges forms a bigon. 
These $2(s-2)$ bigons must not be cells in $\Gamma'$, so each of them covers one of the $2s-4$ vertices 
that were added when inserting $K_4^\times$ in each face of $B_s$.
By the pidgeon-hole principle every bigon covers exactly one of these vertices and vice versa.   Consider the
vertex $z_F$ that was inserted due to some face $F=\{u,v,w\}$ of $B_s$, and the bigon $\calB$ of $\Gamma_B^\shortparallel$
that covers $z_F$.   The $K_4$ formed by $\{u,v,w,z_F\}$ is drawn with at most one crossing in $\Gamma'$,
and since $z_F$ (which is within $\calB$) can have only two uncrossed to $u,v,w$ we have (say)
$(z_F,u)$ crossing edge $(v,w)$.   (This makes $\calB$ the bigon formed by $(v,w)$.)
Let $F'=\{u,v,w\}$ the deg-3 face of $\Gamma_B^\shortparallel$ that corresponds to $F$; edge $(z_F,u)$ then
splits (in $\Gamma$) face $F'$ into two crossed deg-3 cells.   
It follows that all uncrossed cells reside within
bigons, hence no two of them share an uncrossed edge.   Also (as one verifies) all cells of $\Gamma$ have
degree 3, so \ref{it:twoOnCell} and \ref{it:twoOnAdjacentCells} hold and the
drawing is proper-cell-saturated.

\smallskip
For the other bound, take $G_5$, attach a triangle at each base-edge, and duplicate the base-edges to get $G_6$.
(This is the same as drawing $\Gamma_4$ from Theorem~\ref{thm:not3connDrawing_tight} with parallel
edges added to make it triangulated; see also Figure~\ref{fig:not3connGraphNonSimple_tight}.)   
By Theorem~\ref{thm:not3connDrawing_tight} we have $\mu(G_6)\leq \tfrac{n+6}{4}$.
It remains to show any proper-cell 1-planar drawing $\Gamma$ of $G_6$ is proper-cell saturated.   

Let $\Gamma_B^\shortparallel$ be the sub-drawing of $\Gamma$ consisting of all edges of $B_s$, including
parallel edges; as above this is planar and parallel edges form bigons.    These $3s-6$ bigons must each
cover one of the $3s-6$ vertices added for inserted $K_4^\times$'s or triangles.
Let $\Gamma_B^+$ be the sub-drawing of $\Gamma$ where in addition to $\Gamma_B^\shortparallel$
we also include the vertices of inserted $K_4$'s.  As argued above, 
each deg-3 cell $F'$ of $\Gamma_B^\shortparallel$ is hence replaced in $\Gamma_B^+$
by two crossed deg-3 cells.   

Now consider a vertex $z_e$ that is part of an attached triangle and covered by bigon $\calB$.   Assume
for contradiction that an incident edge $(z_e,y)$ is crossed in $\Gamma$, say it intersects edge $(u,w)$.
(See also Figure~\ref{fig:NonSimpleBadDrawing}.)
Since $\calB$ covers no other vertices,  edge $(u,w)$ bounds $\calB$.
The other cell incident to $(u,v)$ was (in $\Gamma_B^\shortparallel$)
a cell $F'=\{u,v,w\}$ which (in $\Gamma_B^+$) is split into two crossed deg-3 cells. 
Edge $(z_e,y)$ enters one of these crossed deg-3 cells, but then it cannot leave again
since it has no other crossing, and it cannot end here since the deg-3 cell contains
no vertices other than $u,v$ (and edges with a common endpoint do not cross).   
This is impossible, so edges incident to $z_e$ are uncrossed, which means that $\calB$
is the bigon formed by $e$.   With this, all cells of $\Gamma$ are again triangles,
all uncrossed cells are within bigons, and one verifies that
\ref{it:twoOnCell} and \ref{it:twoOnAdjacentCells} hold.
\end{proof}

So we get non-trivial matching bounds for proper-cell-maximal 1-planar graphs as long as there are no loops.   The same is
not true for planar graphs.   Consider $K_{2,n}$, with (say) $a,b$ as the two vertices
on the 2-side.   We can add $n$ copies of edge $(a,b)$ to get a planar triangulated drawing.
This is clearly proper-cell-saturated planar (and, up to renaming, the only proper-cell-saturated
planar drawing, so the graph is proper-cell-maximal planar).   But no matching can use more than two edges.

\section{Outlook}
\label{sec:conclusion}

In this paper, we gave tight bounds on the size of a maximum matching
in simple-maximal 1-planar graphs, and specifically in four scenarios,
depending on whether the graph is 3-connected or not and whether the
1-planar drawing is given or not.   We also briefly studied non-simple
graphs, and here again gave tight bounds.

As one open problem,
we would be interested in better tools to explore possible 1-planar
drawings of a graph.   Our arguments in Section~\ref{sec:tightness}
and \ref{sec:non_simple} were tedious due to the need to
argue that {\em all} 1-planar drawings of some graph are saturated.
Can all 1-planar drawings of a graph be described via a small set of
operations, such as the Whitney-flips for planar drawings \cite{Whitney33}?
Can they be stored in a suitable data structure, such as the SPQR-trees
for  planar drawings \cite{DiBattista96b}?     And what conditions
force the 1-planar drawing to be unique, or force an edge to be uncrossed?

Also, there are many other related graph classes that are worth exploring.
What, for example, are lower bounds on the size of matchings in 
simple-maximal 2-planar graphs?

Finally, our bounds rely on the Tutte-Berge formula, and as such, do
not give rise to algorithms to find matchings of the proved size, other
than using general-purpose maximum matching algorithms 
\cite{MV80,Vaz94-matching}.  
For planar and 1-planar graphs with minimum
degree 3, linear-time algorithms have been designed to find matchings that fit the known bounds
on matchings in the class \cite{BiedlKlute20,FRW11}.  Can we develop
similar algorithms for simple-maximal 1-planar graphs and/or simple-saturated 1-planar drawings?

\bibliographystyle{plain}
\bibliography{journal,full,gd,papers}

\begin{appendix}

\section{Missing proofs}

In this appendix, 
we first prove Lemma~\ref{lem:fd}:
In any planar drawing $\Gamma$ we have 
$\sum_d (d-2)f_d=2n-4$, even if $\Gamma$
has loops or parallel edges or is disconnected or $n$ is small.

\begin{proof}
Let $m$ be the number of edges, $f$ be the number of faces, and $\ell$ be the number of
connected components of $\Gamma$.   Assume first that $\ell=1$, so $\Gamma$ is connected.
\emph{Euler's formula} (see e.g.~\cite{Die12}) states that $m=f+n-2$, and this holds
even for $n=1,2$ or in the presence of loops or parallel edges.    Since $\Gamma$ is connected, 
every face $F$ is bounded by a single circuit, and $\deg(F)$ counts the number of edge-incidences
of $F$. Since every edge belongs to two such incidences, we have
$\sum_F \deg(F)=2m=2f+2n-4$, which yields the result after rearranging.

Now we proceed by induction on $\ell$; the base case was covered above.   For $\ell>1$, there exists
a face $F_0$ incident to $k\geq 2$ connected components of $\Gamma$.  We may assume that $F_0$ is the unbounded face, and write
$C_1,\dots,C_k$ for the circuits that bound $F_0$.   Let $\Gamma_i$ (for $i=1,\dots,k$)
be the sub-drawing on or inside $C_i$; this has fewer connected components and so induction applies.
Let $F_i$ be the unbounded face of $\Gamma_i$; it has a connected boundary and so its degree is its number
of edge-incidences.   Therefore $\sum_{i=1}^k \deg(F_i)$ is the number of edge-incidences of $F_0$
and $\deg(F_0)=\sum_{i=1}^k \deg(F_i) + 2k-2$.  
Any  bounded face of $\Gamma$ also exists as bounded face in exactly one $\Gamma_i$, and if we let
$n(\Gamma_i)$ be the number of vertices of drawing $\Gamma_i$, then $\sum_{i=1}^k n(\Gamma_i)=n$.
Putting it all together, we have
\begin{align*}
\sum_d (d-2)f_d 
& = \sum_{\text{face $F$ of $\Gamma$}} (\deg(F){-}2)
= (\deg(F_0){-}2) + \sum_{\text{bounded face $F$ of $\Gamma$}} (\deg(F){-}2)  \\
& = \Big(2k-2+\sum_{i=1}^k \deg(F_i)\Big) - 2 + \sum_{i=1}^k \sum_{\text{bounded face $F$ of $\Gamma_i$}} (\deg(F){-}2)  \\
& = 4k-4 + \sum_{i=1}^k (\deg(F_i) {-}2) + \sum_{i=1}^k \sum_{\text{bounded face $F$ of $\Gamma_i$}} (\deg(F){-}2)  \\
& = 4k-4 + \sum_{i=1}^k \sum_{\text{face $F$ of $\Gamma_i$}} (\deg(F){-}2)  
 = 4k-4 + \sum_{i=1}^k \big( 2n(\Gamma_i)-4 \big) = 2n-4 \quad \qedhere
\end{align*}
\end{proof}

Now we prove Observation~\ref{obs:23}, which characterizes the possible
configurations of a patch of small degree.

\begin{proof}
Fix a face-patch $P$ of degree $d\in \{2,3,4\}$,  and let $L$ be the corresponding cell of $\Gamma_S$.
Since $\comp(\Gamma{\setminus}S)\geq 2$, we have (by  Lemma~\ref{lem:atMostOneCovered}) 
at least two patches, hence at least two cells of $\Gamma_S$, so $L\neq \mathbb{R}^2\setminus \Gamma_S$.
Therefore $P$ (i.e., the boundary of $L$) must include a simple
closed curve $\calC$ that separates some cells covered by $P$ from 
some cells not covered by $P$.   Curve $\calC$ runs along edges of $P$, which
are uncrossed.   Since there are no loops, the number $m_\calC$ of edges
in $\calC$ is at least 2, and all those edges also count for $m_L$ (the
number of edge-incidences of $L$).
Therefore $d=m_L+2\comp(L)-2\geq m_\calC+2\comp(L)-2$.

If $d=2,3$ then $\comp(L) \leq \tfrac{1}{2}(d-m_\calC+2)\leq \tfrac{3}{2}$, and by
integrality $\comp(L)=1$.   So $P$ is a single circuit. If this circuit revisits
a vertex then it contains a loop since it has only two or three edge-incidences, impossible.
So $P$ is a simple $d$-cycle.

For $d=4$, if we have $\comp(L)\geq 2$ then $m_L\geq m_\calC \geq 2$ gives $4=m_L+2\comp(L)-2\geq m_\calC+2\geq 4$,
so equality holds everywhere and $m_L=m_\calC=2$ and $\comp(L)=2$.  So $\calC$ equals one circuit of $P$,
which has two edges, while the other circuit has no edge-incidence, hence is a singleton vertex.

Finally consider the case $d=4$ and $\comp(P)=1$, so $P$ is one circuit with
four edge-incidences.   If some edge $e$ were visited twice by $P$, then $e$
would be incident twice to $L$; we
removed such edges when creating $\Gamma_S$ (see Step~\ref{it:cell_simple} of Definition~\ref{def:GammaS}).
So $P$ is incident to four distinct edges.   If no vertex repeats on $P$ then
it is a simple 4-cycle.   Otherwise only one vertex can repeat  because otherwise
we would have four parallel edges and $L$ would not be a cell of $\Gamma_S$.
\end{proof}

\section{The weights of patches}
\label{sec:weights2}
\label{app:weight_lower}

In this appendix, we prove the bounds of Table~\ref{ta:weight_lower}
(in order of increasing difficulty). 

\begin{claim}
\label{cl:easyBounds}
\label{cl:easy}
$w_\alpha(P)\geq 2-2\alpha$ for any $P\in \calP_3^\nabla$,
and $w_\alpha(P)\geq 4-4\alpha$ for any $P\in \calP^\boxtimes$. 
\end{claim}
\begin{proof}
Any patch $P\in \calP_3^\nabla$ covers exactly one uncrossed cell,
whose weight is $2-2\alpha$ if it is a transfer cell and 2 otherwise.
Any patch $P\in \calP^\boxtimes$ covers exactly four crossed cells,
each of which has weight $1-\alpha$ if it is a transfer cell and 1 otherwise.
\end{proof}

\begin{claim}\label{cl:Pd}
$w_\alpha(P)\geq \min\{0,(2-\alpha)d\}$ for any $P\in \calP_d\setminus P_3^\nabla$.
\end{claim}
\begin{proof}
Let $x_1,\dots,x_k$ be the crossings covered by $P$, and observe that they are not pure-$S$-crossings
since $P\not\in \calP^\boxtimes$.    For $i=1,\dots,k$, set $m_T(x_i)$ to be the number of transfer-edges
that are kite-edges of $x_i$; since at least one endpoint of $x_i$ is not in $S$ we have
$m_T(x_i)\leq 2$.   
Define $w_\alpha(x_i)$ to be the total weight of the four crossed cells incident to $x_i$; we have
$w_\alpha(x)= 4-\alpha{\cdot}m_T(x_i)\geq (2-\alpha)m_T(x_i)$.
So the weight of all crossed cells covered by $P$ is 
$\sum_{i=1}^k w_\alpha(x_i)\geq (2-\alpha)\sum_{i=1}^k m_T(x_i)$.
All uncrossed cells covered by $P$ are not in $\calT^\nabla$, so have non-negative weight, 
so $w_\alpha(P)\geq (2-\alpha)\sum_{i=1}^k m_T(x_i)$.   If $\alpha\leq 2$ then this is
non-negative.   If $\alpha>2$, then this is at least $(2-\alpha)d$, because the circuits that
constitute $P$ contain at most $d$ edge-incidences, and hence $\sum_{i=1}^k m_T(x_i)\leq d$. 
\end{proof}

The bound in Claim~\ref{cl:Pd} is tight; for example consider a patch $P$ bounded by a $d$-cycle
(for even $d$) that covers one vertex $v$ and $d/2$ crossings.    
However, for $d\leq 4$ we cannot construct such a patch without violating \ref{it:fromSimple}
and therefore can find better bounds.   Recall from Claim~\ref{cl:w0_small}
that $Z(P)$ denotes the (non-empty) set of vertices covered by a patch $P\in \calP_d^\odot$ and that (for $d\leq 4$) we have $w_0(P)=4|Z(P)|+2(d-2)$.

\begin{claim} \label{cl:P3}
Assume that $\comp(\Gamma{\setminus}S)\geq 2$ and  there are no loops.
For any patch $P\in \calP_3^\odot$, we have
$w_\alpha(P) \geq \min\{6+3\alpha, 6+6\mychi-\alpha, 10+4\mychi-3\alpha\}$.
\end{claim}
\begin{proof}
Patch $P$ is a simple 3-cycle by Observation~\ref{obs:23};
let its edges be $e_0,e_1,e_2$.   For $i=0,1,2$ let
$e_i=(v_i,v_{i+1})$ (addition modulo 3) and let $L_i$ be the cell incident to $e_i$ that
is covered by $P$.  
Cell $L_i$ has degree 3 and its third corner (other than $v_i,v_{i+1}$) is either a crossing $x_i$, or a
vertex $z_i\not\in S$ since $P\not\in \calP^\nabla$. 
See also Figure~\ref{fig:weight_lower_3}.
Edge $e_i$ is a transfer edge by definition, hence $L_i$ belongs
to $\calT^\circ$ or $\calT^\times$
(it cannot belong to $\calT^\nabla$ since $P\not\in \calP_3^\nabla$),
and $w_\alpha(L_i)=w_0(L_i)\pm \alpha$ depending on whether $L_i$ is uncrossed or not.
So $w_\alpha(P)$ can easily be obtained from $w_0(P)$ once we know how many of $L_0,L_1,L_2$ are crossed.
Recall that $w_0(P)=4|Z(P)|+2\geq 6$ and consider cases.
\begin{itemize}
\vspace*{-2mm}
\item If at least one of $L_0,L_1,L_2$ is uncrossed, then $w_\alpha(P)\geq w_0(P)+\alpha-2\alpha \geq 6-\alpha$.
\item If all of $L_0,L_1,L_2$ are crossed, then $x_0,x_1,x_2$ cannot all be the same crossing, 
	otherwise the kite-edges
	of that crossing would include triangle $e_0,e_1,e_2$, impossible.    So $P$ covers at least two
	crossings (hence eight crossed cells) and $w_0(P)\geq 8$.   This implies
	$w_0(P)\geq 10$ since $w_0(P)=4|Z(P)|+2$ is equivalent to 2 modulo 4.   Therefore $w_\alpha(P)\geq 10-3\alpha$.
\end{itemize}
If \ref{it:multiedgeUncrossed} holds, then we can get a better bound.
\begin{itemize}
\vspace*{-2mm}
\item If all of $L_0,L_1,L_2$ are uncrossed, then $w_\alpha(P)= w_0(P)+3\alpha\geq 6+3\alpha$.
\item If (say) $L_1$ is crossed, then $v_0$ is not an endpoint of $x_1$,  
	otherwise for some $i\in \{1,2\}$ edge $(v_i,v_0)$ would exist crossed at $x_1$ and uncrossed  in $P$
	which contradicts \ref{it:multiedgeUncrossed}.   Therefore $(v_0,z)$ crosses $(v_1,z')$
	for some $z,z'\in Z(P)$, which implies $|Z(P)|\geq 2$ and $w_0(P)\geq 10$.   See also Figure~\ref{fig:weight_lower_3}.
	\begin{itemize}
\vspace*{-2mm}
	\item If both $L_0,L_2$ are uncrossed, then $w_\alpha(P)\geq 10+\alpha\geq \min\{6+3\alpha,14-3\alpha\}$.
	\item If $|Z(P)|\geq 3$, then $w_0(P)\geq 14$ and $w_\alpha(P)\geq 14-3\alpha$.
	\item We claim that one of these cases must hold.  To see this, assume that $L_0$ is crossed, 
		say $(v_0,\hat{z})$ crosses $(v_1,\tilde{z})$ at $x_0$, for some $\hat{z},\tilde{z}\in Z(P)$.
		We have $\tilde{z}\not\in \{z,z'\}$, for otherwise the crossed edge $(v_1,\tilde{z})$ would
		have a parallel edge (in $x_1$ or as a kite-edge of $x_1$),
		contradicting \ref{it:multiedgeUncrossed}.   
		So $|Z(P)|\geq 3$.   \qedhere
	\end{itemize}
\end{itemize}
\end{proof}

\begin{figure}[ht]
\hspace*{\fill}
\subcaptionbox{$\calP_3^\odot$.\label{fig:weight_lower_3}}{\includegraphics[width=0.32\linewidth,page=3]{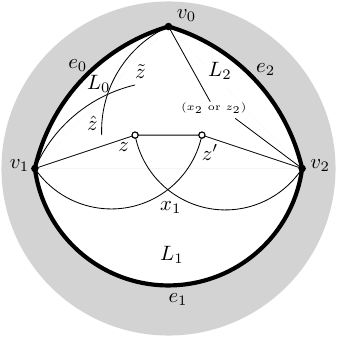}}
\hspace*{\fill}
\subcaptionbox{$\calP_2$; $L_0$ uncrossed.\label{fig:weight_lower_2b}}{\includegraphics[width=0.32\linewidth,page=5]{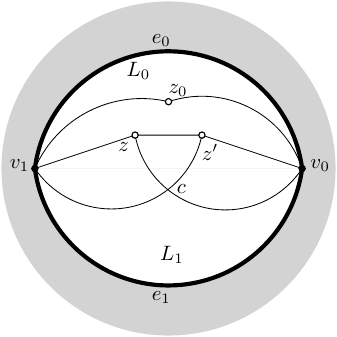}}
\hspace*{\fill}
\subcaptionbox{$\calP_2$; $L_0$ crossed.\label{fig:weight_lower_2a}}{\includegraphics[width=0.32\linewidth,page=4]{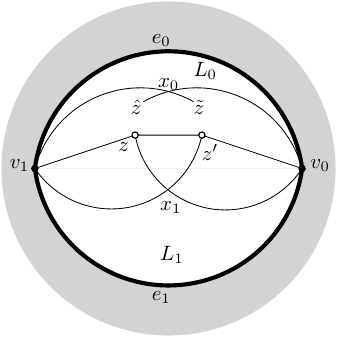}}
\hspace*{\fill}
\caption{Lower-bounding the weight of a patch.}
\label{fig:weight_lower_2}
\end{figure}

For patches in $\calP_2$, we sometimes need \ref{it:fromSimple} even if \ref{it:multiedgeUncrossed} does
not hold; for purposes of analyzing non-simple graphs in Section~\ref{sec:non_simple} we state exactly
what is needed here.

\begin{claim} \label{cl:P2}
Assume that $\comp(\Gamma{\setminus}S)\geq 2$ and  there are no loops. Fix a patch $P\in \calP_2$.
If either \ref{it:fromSimple} holds or $|Z(P)|$ is odd, then
$w_\alpha(P) \geq \min\{4+2\alpha, 12+4\mychi-2\alpha\}$.
\end{claim}
\begin{proof}
By Observation~\ref{obs:23} patch $P$ is a bigon, say edges $e_0,e_1$ both 
connect $v_0$ and $v_1$.  For $i=0,1$ 
let $L_i$ be the cell incident to $e_i$ and covered by $P$. If $L_i$
is crossed then let $x_i$ be its incident crossing, otherwise let $z_i\not\in S$
be the third vertex of $L_i$.
Edge $e_i$ is a transfer edge by definition, hence $L_i$ belongs
to $\calT^\circ$ or $\calT^\times$
and $w_\alpha(L_i)=w_0(L_i)\pm \alpha$ depending on whether $L_i$ is uncrossed or not.
So $w_\alpha(P)$ can easily be obtained from $w_0(P)$ once we know how many of $L_0,L_1$ are crossed.
Recall that $w_0(P)= 4|Z(P)|\geq 4$ and consider cases:
\begin{itemize}
\vspace*{-2mm}
\item If both $L_0,L_1$ are uncrossed, then $w_\alpha(P)\geq w_0(P)+2\alpha\geq 4+2\alpha$.
\item If (say) $L_1$ is crossed, then let $(v_0,z)$,$(v_1,z')$ be the edges that cross at $x_1$
	for some $z,z'\in Z(P)$.  So $|Z(P)|\geq 2$ and $w_0(P)\geq 8$.
	See Figure~\ref{fig:weight_lower_2a} and \ref{fig:weight_lower_2b}.
\begin{itemize}
\vspace*{-2mm}
\item If $L_0$ is uncrossed, then $w_\alpha(P)=w_0(P)+\alpha-\alpha=8  \geq \min\{4+2\alpha,12-2\alpha\}$.
\item If $|Z(P)|\geq 3$ then $w_\alpha(P)\geq w_0(P)-2\alpha\geq 12-2\alpha$.
\item We claim that one of these cases must hold.  Assume not, so $|Z(P)|=2$, which by assumption means
	that \ref{it:fromSimple} holds.   Let the crossing at $x_0$ be between
		$(v_0,\hat{z})$ and $(v_1,\tilde{z})$, for some $\hat{z},\tilde{z}\in Z(P)$.
		By \ref{it:fromSimple} we have $\tilde{z}\neq z'$. 
		By $|Z(P)|=2$ hence $\tilde{z}=z$ and similarly $\hat{z}=z'$.
		Hence $(v_1,\tilde{z})$ and $(z,v_1)$ form a closed curve that is crossed by $(v_0,\hat{z})$
		and hence contains $v_0$ and $\hat{z}$ on opposite sides.
		This is impossible due to (uncrossed) edge $(z',v_0)$ and $\hat{z}=z'$.
\end{itemize}

If \ref{it:multiedgeUncrossed} holds, then we can get a better bound.
\begin{itemize}
\vspace*{-2mm}
\item If $|Z(P)|\geq 4$ then $w_\alpha(P)\geq w_0(P)-2\alpha\geq 16-2\alpha$.
\item If $L_0$ is uncrossed, then $z_0\neq z,z'$, for otherwise $(v_0,z_0)$ or $(v_1,v_0)$ would exist both
	crossed and uncrossed, contradicting \ref{it:multiedgeUncrossed}.   
	Therefore $|Z(P)|\geq 3$ and $w_\alpha(P)\geq 12+\alpha-\alpha=12\geq \min\{4+2\alpha,16-2\alpha\}$.
\item We claim that one of these cases must hold.  To see this, assume that $L_0$ is crossed, 
		say $(v_0,\hat{z})$ crosses $(v_1,\tilde{z})$ at $x_0$, for some $\hat{z},\tilde{z}\in Z(P)$.
		We have $\tilde{z}\not\in \{z,z'\}$, for otherwise there would be parallel copy of 
		crossed edge $(v_1,\tilde{z})$, contradicting \ref{it:multiedgeUncrossed}.   Likewise $\hat{z}\not\in \{z,z'\}$,
		So $|Z(P)|\geq 4$. \qedhere
\end{itemize}
\end{itemize}
\end{proof}

\begin{claim} \label{cl:P4}
Assume that $\comp(\Gamma{\setminus}S)\geq 2$ and  there are no loops. Fix a patch $P\in \calP_4$ and some $\alpha\leq 4$.
If either \ref{it:fromSimple} holds or $\alpha\leq 2$, then
$w_\alpha(P)\geq 0$. 
\end{claim}
\begin{proof}
We are done by Claim~\ref{cl:Pd} if $\alpha\leq 2$, so assume not, which means that \ref{it:fromSimple} holds.
We first dispatch with the easy case where $P$ consists of multiple circuits:
By Observation~\ref{obs:23} patch $P$ uses only two edges and so 
covers at most two transfer cells, none of which is in $\calT^\nabla$.  Also $w_0(P)=4|Z(P)|+4\geq 8$, therefore $w_\alpha(P)\geq 8-2\alpha\geq 0$. 

Now assume that $P$ is a single circuit with
distinct edges $e_0,e_1,e_2,e_3$.    
For $i=0,1,2,3$ let $e_i=(v_i,v_{i+1})$ (addition modulo 4) and let
$L_i$ be the cell incident to $e_i$ and covered by $P$.
If $L_i$ is crossed then let $x_i$ be its incident crossing.
We have $w_\alpha(L_i)\geq w_0(L_i)-\alpha$ since $P\not\in \calP_3^\nabla$, and $w_\alpha(L_i)\geq w_0(L_i)=2$ if $L_i$ is uncrossed.
Now consider cases:
\begin{itemize}
\vspace*{-2mm}
\item Assume first that for every crossing $x$ that is covered by $P$, at most one kite-edge belongs to $P$.
	Then at most one incident cell of $x$ is a transfer cell    and $w_\alpha(x)$
	(the total weight of the four cells incident to $x$) is at most $4-\alpha\geq 0$.
	Any uncrossed cell covered by $P$ has non-negative weight.  Therefore $w_\alpha(P)\geq 0$.
\item So we may assume that there is a crossing  $x$ that is covered by $P$ and where at least two kite-edges
	are transfer edges (hence in $\{e_0,e_1,e_2,e_3\}$).   The kite-edges of $x$ must include $e_i,e_{i+1}$ (for 
	some $i\in \{0,1,2,3\}$, addition modulo 4), and exclude $e_{i+2},e_{i+3}$, otherwise the endpoints of $x$ would be 
	be $\{v_0,v_1,v_2,v_3\}$ and $x$ would be a pure-$S$-crossing, contradicting $P\not\in \calP^\boxtimes$.   
	So $x_i{=}x_{i+1}$ and $x_{i+2},x_{i+3}$ (if they exist) are different from $x_i$.
	See Figure~\ref{fig:weight_lower}, which shows the situation for $i=0$.  We have cases.
\begin{itemize}
\vspace*{-2mm}
\item If both $L_{i+2},L_{i+3}$ are uncrossed, then 
	$w_\alpha(P)\geq w_0(P)-2\alpha \geq 8-2\alpha\geq 0$.
\item If (say) $L_{i+2}$ is crossed and $P$ covers at least one uncrossed cell, then with the eight crossed cells at $x_i\neq x_{i+2}$ we have
	$w_0(P)\geq 10$.  This forces $w_0(P)\geq 12$ since $w_0(P)=4|Z(P)|+4$ is divisible by 4.
	Therefore $w_\alpha(P)\geq 12-3\alpha \geq 0$.
\item If both $L_{i+2},L_{i+3}$ are crossed and $x_{i+2}\neq x_{i+3}$, then $P$ covers at least three crossings, hence $w_0(P)\geq 12$
	and $w_\alpha(P)\geq 12-3\alpha \geq 0$.
\item We claim that one of the above cases must hold.   Assume for contradiction that $L_{i+2}$ and $L_{i+3}$ are both crossed
	and $x_{i+2}{=}x_{i+3}$.
	Then $x_{i+2}{=}x_{i+3}$ uses 
	edge $(v_{i+2},v_i)$, but a copy of $(v_i,v_{i+2})$  also participates in the (different) crossing $x_{i}{=}x_{i+1}$.
	This contradicts \ref{it:fromSimple}. \qedhere
\end{itemize}
\end{itemize}
\end{proof}

\begin{figure}[ht]
\hspace*{\fill}
\subcaptionbox{$\calP_4$; $P$ is a simple 4-cycle.}{\includegraphics[width=0.40\linewidth,page=1]{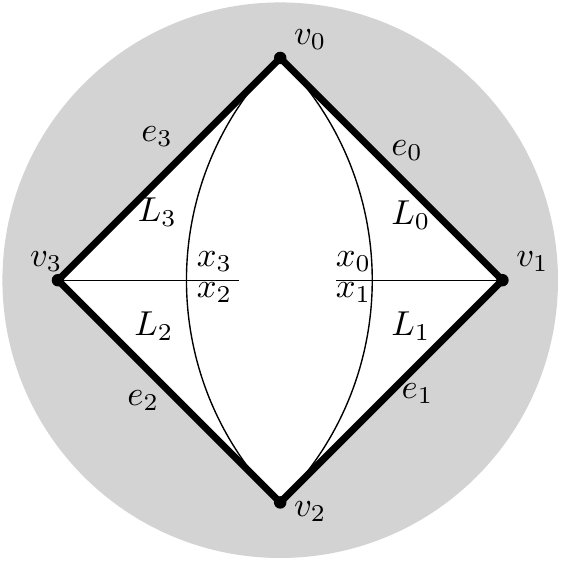}}
\hspace*{\fill}
\subcaptionbox{$\calP_4$; $P$ repeats a vertex.~}{\includegraphics[width=0.40\linewidth,page=2]{weight_lower.pdf}}
\hspace*{\fill}
\caption{Lower-bouding the weight of a patch in $\calP_4$.}
\label{fig:weight_lower}
\end{figure}
\end{appendix}

\end{document}